\documentclass[11pt]{amsart}
\usepackage{todonotes}
\usepackage{xifthen}
\usepackage[T1]{fontenc}
\usepackage[english]{babel}
\usepackage{amsmath, amssymb, amsthm}  
\usepackage{mathrsfs}          
\usepackage{amstext, amsfonts, a4}
\usepackage{bbm}
\usepackage{wasysym}
\usepackage{caption}
\usepackage{lmodern}
\usepackage{animate} 

\usepackage[
backend=biber,        
sorting=nty,          
maxcitenames=2,
mincitenames=1,
maxbibnames=1000,
doi=false,
eprint=true,
url=false,
giveninits=true
]{biblatex}

\usepackage[colorlinks=true]{hyperref}
\hypersetup{urlcolor=blue, citecolor=blue, linkcolor=blue}

\usepackage{color,transparent}
\usepackage{graphics,graphicx,subfigure}
\usepackage{csquotes}
\usepackage{comment}
\usepackage{float} 
\graphicspath{{./images/}}

\theoremstyle{plain}
\newtheorem{theorem}{Theorem}[section] 
\newtheorem{proposition}[theorem]{Proposition}
\newtheorem{lemma}[theorem]{Lemma}
\newtheorem{corollary}[theorem]{Corollary}
\newtheorem{definition}[theorem]{Definition}
\newtheorem{remark}[theorem]{Remark}
\newtheorem{hypothesis}[theorem]{Hypothesis}

\newcommand{\bs}{\symbol{92}}
\def\ogg~{{\rm \og}}   

\def\N{{\mathbb N}}    
\def\Z{{\mathbb Z}}     
\def\R{{\mathbb R}}    
\def\Q{{\mathbb Q}}    
\def\C{{\mathbb C}}    

  \def\cG{{\mathcal G}}  \def\cS{{\mathcal S}}     \def\cT{{\mathcal T}}          \def\cV{{\mathcal V}}    \def\cQ{{\mathcal Q}} \def\cW{{\mathcal W}} \def\cF{{\mathcal F}}  \def\cL{{\mathcal L}}

\newcommand{\erf}{\operatorname{erf}}

\newcommand{\GL}{\operatorname{GL}}             

\DeclareMathOperator{\idty}{Id}
\DeclareMathOperator{\v1}{V1}

\newcommand\Res{\operatorname{Res}}

\hypersetup{
	pdfauthor={C. Tamekue, Y. Chitour, and D. Prandi},
	pdfsubject={Paper manuscrit},
	pdftitle={On the mathematical replication of the MacKay effect},
	pdfkeywords={	Control in neuroscience, Exact controllability, Neural field model, Amari-type equation, Visual illusions and perception, MacKay effect, Spatially forced pattern forming system.}
}

\begin{document}
	\title[A mathematical model of the visual MacKay effect]{A mathematical model of the visual MacKay effect}
	\thanks{This work was supported by ANR-20-CE48-0003 and ANR-11-IDEX-0003, Objet interdisciplinaire H-Code. The work of the first author was supported by a grant from the bourse de th\`eses ``Jean-Pierre Aguilar''. 
	}
	
	\author{Cyprien Tamekue}
	\address{Université Paris-Saclay, CNRS, CentraleSupélec, Laboratoire des signaux et systèmes, 91190, Gif-sur-Yvette, France}
	\email{cyprien.tamekue@centralesupelec.fr}
	
	\author{Dario Prandi}
	\address{Université Paris-Saclay, CNRS, CentraleSupélec, Laboratoire des signaux et systèmes, 91190, Gif-sur-Yvette, France}
	\email{dario.prandi@centralesupelec.fr}
	
	\author{Yacine Chitour}
	\address{Université Paris-Saclay, CNRS, CentraleSupélec, Laboratoire des signaux et systèmes, 91190, Gif-sur-Yvette, France}
	\email{yacine.chitour@centralesupelec.fr}
	
	\begin{abstract}
This paper investigates the intricate connection between visual perception and the mathematical modeling of neural activity in the primary visual cortex (V1). The focus is on modeling the visual MacKay effect  [D. M. MacKay, Nature, 180 (1957), pp. 849–850]. While bifurcation theory has been a prominent mathematical approach for addressing issues in neuroscience, especially in describing spontaneous pattern formations in V1 due to parameter changes, it faces challenges in scenarios with localized sensory inputs. This is evident, for instance, in Mackay's psychophysical experiments, where the redundancy of visual stimuli information results in irregular shapes, making bifurcation theory and multi-scale analysis less effective. To address this, we follow a mathematical viewpoint based on the input-output controllability of an Amari-type neural fields model. In this framework, we consider sensory input as a control function, a cortical representation via the retino-cortical map of the visual stimulus that captures its distinct features. This includes highly localized information in the center of MacKay's funnel pattern ``MacKay rays''. From a control theory point of view, the Amari-type equation's exact controllability property is discussed for linear and nonlinear response functions. For the visual MacKay effect modeling, we adjust the parameter representing intra-neuron connectivity to ensure that cortical activity exponentially stabilizes to the stationary state in the absence of sensory input. Then, we perform quantitative and qualitative studies to demonstrate that they capture all the essential features of the induced after-image reported by MacKay.
			\textbf{MSCcodes.}  93C20, 92C20, 35B36, 45A05, 45K05.
	\end{abstract}

%
	
	\maketitle
	
	\tableofcontents

\section{Introduction}\label{s::intro}

A simple yet profound question that can arise in humans daily is how we can control the complexities around us, whether they are mathematical equations or even our perceptions of reality. With its intricate network of neurons, the brain is a prime example of a complex system that can be understood through the lens of control theory. Neuroscientists and psychophysics researchers have been fascinated by intriguing phenomena like the MacKay effect \cite{mackay1957,mackay1961} during which people experience visual illusions. In contrast to a visual hallucination, which refers to the perception of an image that does not exist or is not present in front of the person who has experienced it, a visual illusion often occurs when external stimuli trick our brain into perceiving something differently from its actual state. 

In the last decades, investigations of mechanisms underlying the spontaneous perception of visual hallucination patterns have been widely undertaken in the literature using neural dynamics in the primary visual cortex (hereafter referred to as V1) when its activity is due solely to the random firing of its spiking neurons, that is in the absence of sensory inputs from the retina, \cite{bressloff2001,ermentrout1979,golubitsky2003,tass1995,bressloff2002}. In their seminal work \cite{ermentrout1979}, by using bifurcation techniques near a Turing-like instability, Ermentrout and Cowan found that the $2$-dimensional two-layer neural fields equation modelling the average membrane potential of spiking neurons in V1 derived by Wilson and Cowan in \cite{wilson1973} is sufficient to theoretically describe the spontaneous formation (i.e., in the absence of visual sensory inputs) of some geometric patterns (horizontal, vertical and oblique stripes, square, hexagonal and rectangular patterns, etc.) in V1. These patterns result from activity spreading over this brain area and correspond to states of highest cortical activities. When we transform these patterns by the inverse of the retino-cortical map from V1 onto the visual field \cite{tootell1982,schwartz1977}, what we obtain in the retina in terms of images are geometric visual hallucinations. They correspond to some of the form constants that Klüver had meticulously classified \cite{kluver1966}, mainly those contrasting regions of white and black (funnel, tunnel, spiral, checkerboard, phosphenes). Therefore, the neural dynamic equation used to model the cortical activity in V1 combined with the bijective nonlinear retino-cortical mapping between the visual field and V1 predicts the geometric forms of hallucinatory patterns. 

While spontaneous patterns that emerge in V1 give us insight into the underlying architecture of the brain’s neural network, little is known about how precisely the intrinsic circuitry of the primary visual cortex generates the patterns of activity that underlie the visual illusions induced by visual stimuli from the retina. We study in this paper the interaction between retinal stimulation by geometrical patterns and the cortical response in the primary visual cortex, focusing on the MacKay effect \cite{mackay1957} replication using control of Amari-type equation. As a control term, we consider the sensory input from the retina modelling the V1 representation via the retino-cortical map of the visual stimulus used in this intriguing visual phenomenon.

\subsection{The visual MacKay effect}

\begin{figure}
	\centering
	\includegraphics[width=.32\linewidth]{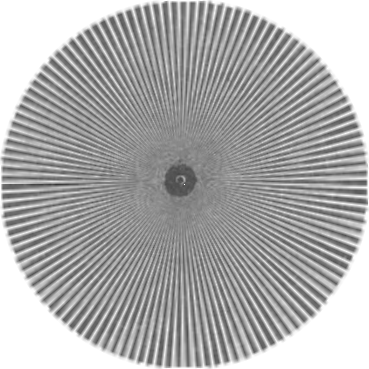}\hspace{1em}
	\includegraphics[width=.32\linewidth]{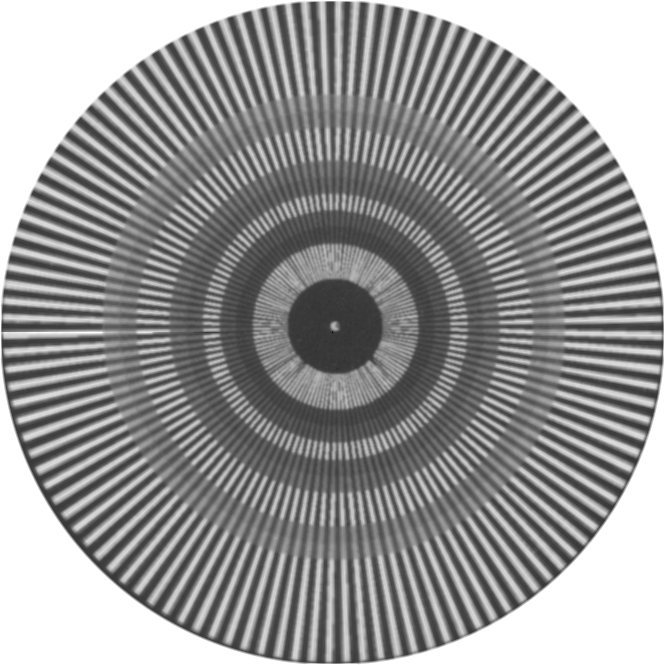}
	\caption{The MacKay effect \cite{mackay1957}, showcasing the illusion induced by the stimulus on the left, referred to as ``MacKay rays''. This stimulus leads to an illusory perception of concentric rings superimposed in its background, as illustrated on the right. To see the illusory contours, look at the centre of the black circle in the image on the \emph{left}. The adaptation of this figure is based on the original representation from \cite[Fig.~1]{mackay1957} and \cite[Fig.~1b]{zeki1993}.}
	\label{fig:mackay}
\end{figure}

Around 1960, Donald MacKay made notable observations on the after-effects of visual stimulation using regular geometrical patterns containing highly redundant information. He associated this phenomenon, now known as the “MacKay effect”, with a specific region of the visual cortex that potentially benefits from such redundancy \cite{mackay1957}. The psychophysical experiments presented in this paper demonstrate that when a highly redundant visual stimulus, such as a funnel pattern (fan shapes), is presented at the center of the stimulus, an accompanying illusory tunnel pattern (concentric rings) emerges in the visual field, superimposed onto the stimulus pattern (see Fig.~\ref{fig:mackay}). 

Notably, the distance from the pattern to the retina or the illumination does not significantly affect these more intricate phenomena. For most observers, the illusory contours in the background of the afterimage rotate rapidly at right angles to the stimulus pattern, either clockwise or counterclockwise.
Similarly, many observers perceive an illusory funnel pattern superimposed in the afterimage background when viewing a tunnel pattern as that of Figure~\ref{fig::MacKay target}, the right panel. In both cases, observers often note rapidly fluctuating sectors, again rotating either clockwise or counterclockwise. Notably, the stimulus pattern does not need to fill the entire visual field; a portion of the stimulus is sufficient to generate a corresponding afterimage in the same portion. However, in both cases, the nervous system tends to prefer the direction perpendicular to the regular contours of the visual stimulus. The present paper proposes to attribute this preference to the retino-cortical map, resulting in induced afterimages of superimposed patterns of horizontal and vertical stripes in V1.

\subsection{Strategy of study and presentation of our results }
In Section~\ref{ss::SS}, we expound on our strategy to theoretically replicate the visual MacKay effect that we recalled in the previous section. Subsequently, in Section~\ref{ss::presentation of results}, we present and discuss our findings.

This work originated in \cite{tamekue2022}, where we developed a novel approach to describe the MacKay effect specifically, associated with a funnel pattern containing high localized information in the center of the image (\textit{created by the very fast alternation of black and white rays}) \cite{mackay1957}. Instead of relying on traditional mathematical tools such as bifurcation analysis, perturbation theories, or multi-scale analysis, commonly used to address neuroscience questions, we sought alternative methods via control of the Amari-type neural fields.


\begin{figure}
	\centering
	\includegraphics[width=.31\linewidth]{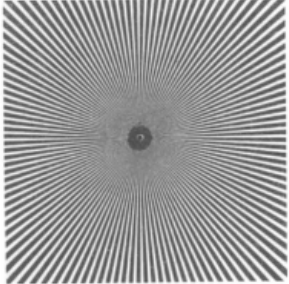}
	\hspace{2em}
	\includegraphics[width=.3\linewidth]{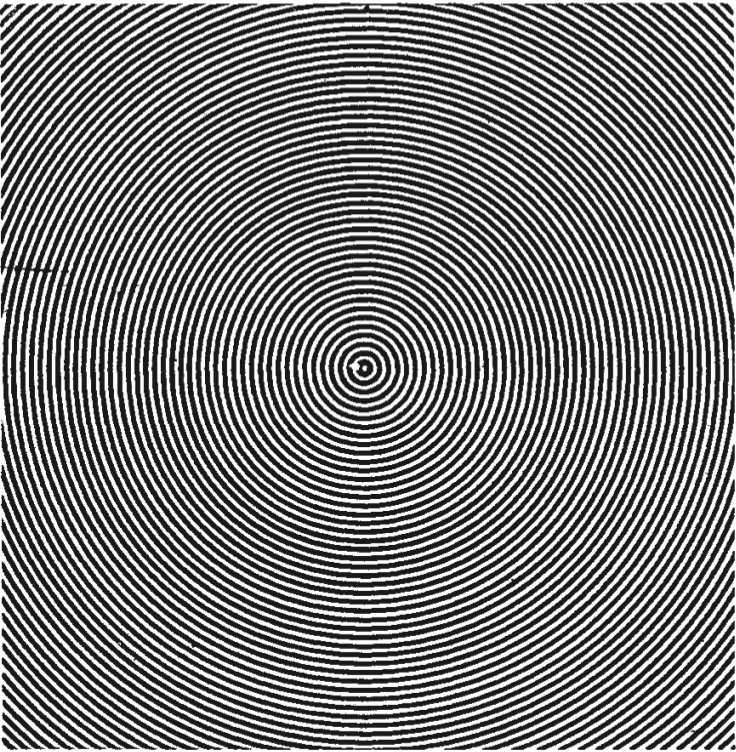}
	\caption{The stimuli used in \cite{mackay1957}. \emph{Left:} ``MacKay-rays''.
		\emph{Right:} ``MacKay-target''. }
	\label{fig::MacKay target}
\end{figure}

Indeed, these classical mathematical tools are highly suitable for describing phenomena like spontaneous geometric visual hallucinations that emerge in the visual field due to sudden qualitative changes in specific physiological parameters \cite{bressloff2001,ermentrout1979,golubitsky2003,tass1995}. They also prove effective in understanding sensory-driven and self-organized cortical activity interactions when the visual stimulus exhibits regular shape and complete distribution across the visual field, with symmetry respecting a subgroup of the Euclidean group \cite{nicks2021}. In simple terms, these tools are appropriate when dealing with equations that exhibit complete equivariance (commutation) with respect to a given group, typically the Euclidean group. However, the original MacKay stimulus, known as the ``MacKay rays'' (refer to Fig.~\ref{fig::MacKay target}), consists of funnel patterns with high localized information in the center. As a result, the Euclidean symmetry of the V1 representation via the retino-cortical map  of the funnel patterns is broken (see Section~\ref{ss:: Binary pattern}). This corresponds to the fact that the funnel pattern is invariant by dilations in the visual plane, while the ``MacKay rays'' are not.
As a consequence, the Amari-type neural field describing the V1 activity induced by ``MacKay rays'' does not present any symmetry.
Accordingly, neither the works of \cite{bressloff2001,ermentrout1979} nor \cite{nicks2021} can be directly employed to describe these complex visual phenomena. In particular, the work \cite{nicks2021} theoretically replicate a variant of the MacKay effect where the visual stimulus is not the ``MacKay rays'' nor the ``MacKay target'' (see Figure~\ref{fig::MacKay target}, right) but a regular (symmetric with respect to some subgroups of the plane Euclidean group) funnel or tunnel patterns, which is fully distributed in the visual field.

\subsubsection{Strategy of study}\label{ss::SS}

In our study, we begin by assuming that neurons in V1 are interconnected in a homogeneous and isotropic manner. Accordingly, we employ the following Amari-type equation \cite[Eq. (3)]{amari1977} to describe the average membrane potential of V1 spiking neurons that take into account the sensory input from the retina:
\begin{equation}\label{eq::NF-intro}\tag{NF}
	\partial_t a = - a +\mu \omega\ast f(a) + I.
\end{equation}
Here $a:\R_{+}\times\mathbb R^2\to \R$ is a function of time $t\in\R_{+}$ and the position $x\in\R^2$, the sensory input $I$ represents the projection of the visual stimulus into V1 by the retino-cortical map. The connectivity kernel $\omega(x,y)=\omega(|x-y|)$ models the strength of connections between neurons located at positions $x\in \mathbb{R}^2$ and $y\in\mathbb{R}^2$. The function $f$ captures the nonlinear response of neurons after activation, while the parameter $\mu>0$ characterizes the intra-neural connectivity. The symbol $\ast$ denotes spatial convolution, as defined in \eqref{eq::spatial convolution} below.


We observe that it would be natural to model V1 as a bounded domain $\Omega$, instead of $\mathbb{R}^2$. We stress that due to the strongly localized connectivity kernel that we consider, our results could in principle be extended to this case, at least to describe neuronal activity sufficiently far from the boundary of $\Omega$.
Moreover, although the more plausible biological neuronal dynamics in V1 involve considering the orientation preferences of ``simple cells'' as done in \cite[Eq. (1)]{bressloff2001} when describing contoured spontaneous cortical patterns, we neglect the orientation label entirely and focus on equation \eqref{eq::NF-intro}. This simplification is motivated by the fact that equation \eqref{eq::NF-intro} is sufficient for describing spontaneous funnel and tunnel patterns, and we expect it also to be suitable for describing psychophysical experiments involving these patterns.


In psychophysical experiments, observers perceive an illusory afterimage in their visual field when viewing the visual stimulus, and this afterimage persists for a few seconds. Therefore, describing these intriguing visual phenomena in V1 using equation \eqref{eq::NF-intro} relies on explicitly studying the map $\Psi$, which associates the sensory input $I$ with its corresponding stationary output $\Psi(I)$. The stationary output represents the stationary solution of equation \eqref{eq::NF-intro} for a given $I$. Our goal is to prove that the cortical activity $a(t,\cdot)$, which is the solution of equation \eqref{eq::NF-intro}, exponentially stabilizes towards $\Psi(I)$ as $t\to+\infty$. Then, we perform qualitative and quantitative study of this stationary state in a convenient space.


Before performing the asymptotic study (qualitatively and quantitatively) of the input to stationary output map for modelling the visual MacKay effect, we investigate the exact controllability properties of the Amari-type control system \eqref{eq::NF-intro} where we interpret the sensory input $I$ as a distributed control over $\R^2$ that we use to act on the system state modelled by the cortical activity $a$.
In that direction, we prove that \eqref{eq::NF-intro} is exactly controllable, in the sense explained previously, except for certain particular functional frameworks.


\subsubsection{Presentation of results}\label{ss::presentation of results}
In our previous paper \cite{tamekue2022}, we established that to accurately model the visual stimuli used, for instance, in the MacKay effect associated with the ``MacKay rays'' visual stimulus, it is crucial to consider the highly localized information present in the center of the funnel patterns as created by the very fast alternation of black and white rays.
This observation arises from the underlying Euclidean symmetry of V1, which imposes restrictions on the geometric shapes of sensory inputs capable of inducing cortical illusions in V1. Interestingly, this mathematical evidence supports the observation previously made by MacKay in paragraph 2 of \cite{mackay1957}: \textit{``$[\cdots]$ in investigations of the visual information system, it might be especially interesting to observe the effect of highly redundant information patterns since the nervous system might conceivably have its own ways of profiting from such redundancy $[\cdots]$''}.

To model the redundant information in the center of the funnel pattern created by the very fast alternation of black and white rays, we employed the characteristic function of a small disk in the center of the visual field as a control function. Through numerical simulations, we suggested that equation \eqref{eq::NF-intro}, together with an \textit{odd} sigmoidal response function, successfully reproduces the MacKay effect associated with the ``MacKay rays''. We employed a similar approach to reproduce the MacKay effect associated with the ``MacKay target'', except that the control function was chosen as the characteristic function of two symmetric rays converging towards the center of the image (refer to Remark~\ref{rmk::comments on the MacKay target modelling} for more details on this modelling).

Having established that equation \eqref{eq::NF-intro}, with appropriate modelling of MacKay visual stimuli, reproduces this phenomenon, our next objective was to provide a mathematical proof of the numerical results obtained in \cite{tamekue2022}. Therefore, in  \cite{tamekue2023}, we discovered that the linearized version of \eqref{eq::NF-intro} is sufficient to describe and replicate the MacKay effect, indicating that the nonlinear nature of the response function does not play a role in its reproduction. Specifically, the saturation effect only serves to dampen high oscillations that can occur in V1. 

In this paper, we provide a mathematical modelling of the visual MacKay effect using Equation~\eqref{eq::NF-intro}, employing complex and harmonic analysis tools and sharp inequality estimates. Specifically, we exploit {Fourier analysis}. 

To the authors' knowledge, the only attempt to theoretically replicate the MacKay-like phenomenon using neural fields equations has been undertaken by Nicks \textit{et al.} \cite{nicks2021}. They employed a model of cortical activity in V1, which included spike-frequency adaptation (SFA) of excitatory neurons, and utilized bifurcation and multi-scale analysis near a Turing-like instability to describe the MacKay-type effect associated with a fully distributed state-dependent sensory input representing cortical representations of funnel and tunnel patterns. By assuming a balanced condition\footnote{A kernel $\omega$ of ``Mexican-hat'' type distribution satisfies the balanced condition (between excitatory and inhibitory neurons) if its Fourier transform at zero equals $0$.} on the interaction kernel, they derived a dynamical equation for the amplitude of the stationary solution near the critical value $\mu_c$ (see Equation~\eqref{eq::parameter mu_c} below) of the parameter $\mu$ where spontaneous cortical patterns emerge in V1. Their theoretical results do not apply to localized inputs, as those employed by MacKay \cite{mackay1957}.

In the present study, to address the specificity of the sensory inputs utilized in these psychophysical experiments (i.e., the highly localized information in MacKay's stimuli), we rely on a central assumption regarding the range of parameter $\mu$. We assume that $\mu$ is smaller than the threshold $\mu_c$, given as follows,
\begin{equation}\label{eq::parameter mu_c}
	\mu_c:= 
	\frac{1}{f'(0)\max\limits_{\xi\in\mathbb{R}^2}\hat\omega(\xi)},
\end{equation}
corresponding to the value of $\mu$ where cortical patterns spontaneously emerge in V1, \cite{ermentrout1979,bressloff2001}.

Finally, we stress once again that our focus lies in assessing the qualitative concordance between the outputs of the proposed models and the observed human perceptual response to these illusions reported by MacKay. It is imperative to emphasize that this inquiry is qualitative, demonstrating the potential utility of Amari-type dynamics in reproducing the perceptual distortions elicited by certain visual illusions.

\subsection{Structure of the paper}
The remaining of the paper is organized as follows: Section~\ref{s::GN} begins by introducing the general notations that will be utilized throughout the paper. 
We present assumptions on model parameters used in Equation~\eqref{eq::NF-intro} in Section~\ref{ss::Assumption on parameters}, and we define a binary pattern necessary to represent cortical activity in terms of white and black zones in Section~\ref{ss:: Binary pattern}. In Section~\ref{ss::WP}, we recall some preliminary results about the well-posedness of equation \eqref{eq::NF-intro}, and in Section~\ref{s::control Amari equation}, we discuss the exact controllability properties of Equation~\eqref{eq::NF-intro}. Using equation \eqref{eq::NF-intro}, in Section~\ref{s::MacKay mathematical interpretation}, we investigate the theoretical replication of the visual MacKay effect. In Section~\ref{ss::numerical results for the MacKay effect}, we present numerical results to bolster our theoretical study. Finally, we provide in the Appendix some technical Theorems that serve as complement results.

\subsection{General notations}\label{s::GN}

Unless otherwise stated, $p$ will denote a real number satisfying $1\le p\le\infty$, and $q$ will denote the conjugate to $p$ given by $1/p+1/q = 1$. We adopt the convention that the conjugate of $p=1$ is $q = \infty$ and vice-versa.

For $d\in\{1,2\}$, we denote by $L^p(\R^d)$ the Lebesgue space of class of real-valued measurable functions $u$ on $\R^d$ such that $|u|$ is integrable over $\R^d$ if $p<\infty$, and $|u|$ is essentially bounded over $\R^d$ when $p=\infty$. We endow these spaces with their standard norms
$$\|u\|_p^p = \int_{\R^d}|u(x)|^pdx,\qquad\mbox{and}\qquad \|u\|_\infty = \operatorname{ess}\sup_{x\in\R^d}|u(x)|.$$

We let $X_p:=C([0,\infty);L^p(\R^d))$ be the space of all real-valued functions $u$ on $\R^d\times[0,\infty)$ such that, $u(x,\cdot)$ is continuous on $[0,\infty)$ for $\mbox{a.e.},\; x\in\R^d$ and $u(\cdot,t)\in L^p(\R^d)$ for every $t\in[0,\infty)$. In $X_p$, we will use the following norm $	\|u\|_{L_t^\infty L_x^p} = \sup\limits_{t\ge 0}\|u(\cdot,t)\|_p$.

For $x\in\R^2$, we denote by $|x|$ its Euclidean norm, and the scalar product with $\xi\in\R^2$ is defined by $\langle x,\xi\rangle=x_1\xi_1+x_2\xi_2$.

We let $\cS(\R^d)$ be the Schwartz space of rapidly-decreasing $C^\infty(\R^d)$ functions, and $\cS'(\R^d)$ be its dual space, i.e., the space of tempered distributions. Then, $\cS(\R^d)\subset L^p(\R^d)$ and $L^p(\R^d)\subset\cS'(\R^d)$ continuously. The Fourier transform of $u\in \cS(\R^d)$ is defined by
\begin{equation}\label{eq::Fourier transform in S}
	\widehat{u}(\xi):= \cF\{u\}(\xi)=\int_{\R^d} u(x)e^{-2\pi i\langle x,\xi\rangle}dx,\qquad\qquad\forall\xi\in\R^d.
\end{equation}
We highlight that,
for $1\le p\le 2$, the above definition can be continuously extends to function $u\in L^p(\R^d)$ by density and Riesz-Thorin interpolation theorem. Whereas one can extend the above by duality to $\cS'(\R^d)$.
We recall that $\cF$ is a linear isomorphism from $\cS(\R^d)$ to itself and from $\cS'(\R)$ to itself.

The spatial convolution of two functions $u\in L^1(\R^d)$ and $v\in L^p(\R^d)$, $1\le p\le\infty$ is defined by
\begin{equation}\label{eq::spatial convolution}
	(u\ast v)(x) = \int_{\R^d}u(x-y)v(y)dy,\qquad x\in\R^d.
\end{equation}

Finally, the following notation will be helpful: if $F$ is a real-valued function defined on $\R^2$, we use  $F^{-1}(\{0\})$ to denote the zero level-set of $F$.

\section{Assumptions on parameters and binary representation of patterns}\label{s::model parameters}
We present in this section assumptions on model parameters used in Equation~\eqref{eq::NF-intro} and the definition of a binary pattern necessary to represent cortical activity in terms of white and black zones.
\begin{figure}
	\centering
	\includegraphics[width=.43\linewidth]{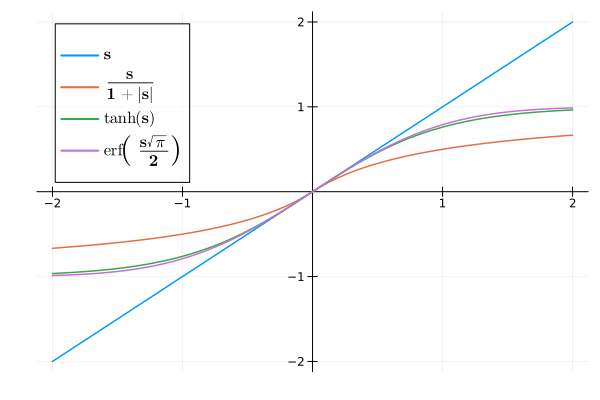}\hspace{0.1em}
	\includegraphics[width=.52\linewidth]{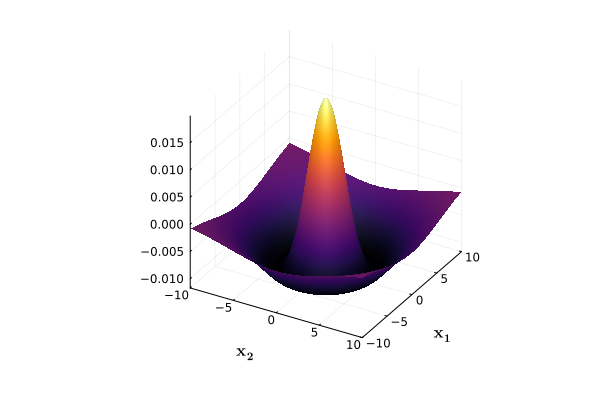}
	\caption{Possible response functions on the \textit{left} where $\erf$ is the Gauss error function, and on the \textit{right} a $2$D DoG kernel $\omega$. Here, $\kappa=2$, $\sigma_1=2$, and $\sigma_2=4$.}
	\label{fig:response function}
\end{figure}
\subsection{Assumption on parameters for Amari-type equation}\label{ss::Assumption on parameters}
We assume that the response function $f$ belongs to the class $C^2(\R)$, is non-decreasing, satisfies $f(0) = 0$, $f'(0) = \max_{s\in\R} f'(s)$, and $f''$ is also bounded so that $f'$ is Lipschitz continuous. Please refer to Figure~\ref{fig:response function} (image on the left) for an example of a response function. 

Unless otherwise stated, we consider $f$ to be linear or a nonlinear and sigmoid function, such that $\|f\|_\infty = 1$, $f'(0) = 1$. {This is without loss of generality} since, as long as $f'(0)\neq 0$, we can always define a sigmoid function $\widetilde{f}(s) = f(\lambda s)/\|f\|_\infty$ with $\lambda=\|f\|_\infty/f'(0)$ and $s\in\R$.
%

The interaction kernel $\omega$ is chosen to be the following difference of Gaussians (DoG)
\begin{equation}\label{eq::connectivity}
	\omega(x) = [2\pi\sigma_1^2]^{-1}e^{-\frac{|x|^2}{2\sigma_1^2}}-\kappa[2\pi\sigma_2^2]^{-1}e^{-\frac{|x|^2}{2\sigma_2^2}},\qquad x\in\R^2,
\end{equation}
where $\kappa>0$, $0<\sigma_1<\sigma_2$, and $\sigma_1\sqrt{\kappa}<\sigma_2$. It is worth noting that this choice of interaction kernel aligns with the framework employing Equation~\eqref{eq::NF-intro} to generate spontaneous cortical patterns in V1.

In particular, $\omega$ is homogeneous and isotropic with respect to the spatial coordinates. It solely depends on the Euclidean distance between neurons and exhibits rotational symmetry. The first (positive) Gaussian in \eqref{eq::connectivity} describes short-range excitation interactions, while the second (negative) Gaussian represents long-range inhibition interactions between neurons in V1.


It is important to observe that $\omega(x) = \omega(|x|)$ and that $\omega$ belongs to the Schwartz space $\cS(\R^2)$, implying that $\omega\in L^p(\R^2)$ for all real numbers $1\le p\le\infty$. The Fourier transform of $\omega$ can be explicitly expressed as
\begin{equation}\label{eq::Fourier transform of the kernel omega}
	\widehat{\omega}(\xi)= e^{-2\pi^2\sigma_1^2|\xi|^2}-\kappa e^{-2\pi^2\sigma_2^2|\xi|^2},\qquad\qquad\qquad\forall\xi\in\R^2,
\end{equation}
and $\widehat{\omega}$ reaches its maximum at every vector $\xi_c\in\R^2$ such that 
\begin{equation}\label{eq::maximum of omegahat}
	|\xi_c|= q_c := \sqrt{\frac{\log\left(\frac{\kappa\sigma_2^2}{\sigma_1^2}\right)}{2\pi^2(\sigma_2^2-\sigma_1^2)}}\qquad\qquad\text{and}\qquad\qquad \max\limits_{r\ge 0}\widehat{\omega}(r) = \widehat{\omega}(q_c).
\end{equation}
{Finally}, the explicit expression for the $L^1$-norm of $\omega$ is given by
\begin{equation}\label{eq::L^1-norm of omega}
	\|\omega\|_1 = (1-\kappa)+2\left(\kappa e^{-\frac{\Theta^2}{2\sigma_2^2}}-e^{-\frac{\Theta^2}{2\sigma_1^2}}\right)\qquad\quad\text{with}\qquad\quad \Theta:= \sigma_1\sigma_2\sqrt{\frac{2\log\left(\frac{\sigma_2^2}{\kappa\sigma_1^2}\right)}{\sigma_2^2-\sigma_1^2}}.
\end{equation}

We emphasize that the kernel $\omega$ does not necessarily satisfy the \textit{balanced}\footnote{For a homogeneous NF equation (i.e., when $I = 0$), this condition ensures the existence of a unique stationary state $a_0=0$ even if $f(0)\neq 0$. It was assumed, for instance, in \cite{nicks2021} for deriving the amplitude equation of the stationary state near the bifurcation point $\mu_c$.} condition $\widehat{\omega}(0)=0$ between excitation and inhibition. However, this condition is met when $\kappa = 1$.



\subsection{Binary representation of patterns}\label{ss:: Binary pattern}

Let us start this section by briefly recalling the retino-cortical map that can be found in \cite{schwartz1977,bressloff2001}. Let $(r,\theta)\in[0,\infty)\times[0,2\pi)$ denote polar coordinates in the visual field (or in the retina) and $(x_1,x_2)\in\R^2$ Cartesian coordinates in $\v1$. The retino-cortical map (see also \cite{tamekue2022} and references within) is analytically given by 
\begin{equation}\label{eq::retino-cortical}
	r e^{i\theta}  \mapsto (x_1,x_2):=\left( \log r, \theta \right).
\end{equation}

Due to the retino-cortical map~\eqref{eq::retino-cortical}, funnel and tunnel patterns are respectively given in Cartesian coordinates $x:=(x_1,x_2)\in\R^2$ of V1 by
\begin{equation}\label{eq::funnel and tunnel patterns}
	\hspace{-0.3cm}  P_F(x) = \cos(2\pi\lambda x_2),\qquad\qquad P_T(x) = \cos(2\pi\lambda x_1),\quad\lambda>0.
\end{equation}
This choice is motivated by analogy with the (spontaneous) geometric hallucinatory patterns described in \cite{ermentrout1979} and \cite{bressloff2001}. Given the above representation of funnel and tunnel patterns in cortical coordinates, to see how they look in terms of images, we represent them as contrasting white and black regions, see Figure.~\ref{fig::FT pattern}. More precisely, define the binary pattern  $B_h$ of a function $h:\R^2\to\R$ by
\begin{equation}
	B_h(x) =
	\begin{cases}
		0, \quad \text{if } h(x)> 0 \quad\text{(black)}\\
		1, \quad \text{if } h(x)\le 0 \quad\text{(white)}.\\
	\end{cases}
\end{equation}
It follows that $B_h$ is essentially determined by the zero level-set of $h$.
Since stimuli involved in the MacKay effect are binary patterns, our strategy in describing these phenomena consists of characterising the zero-level set of output patterns. That is,
we are mainly devoted to studying the qualitative properties of patterns by viewing them as binary patterns.
\begin{figure}
	\centering
	\includegraphics[width=.485\linewidth]{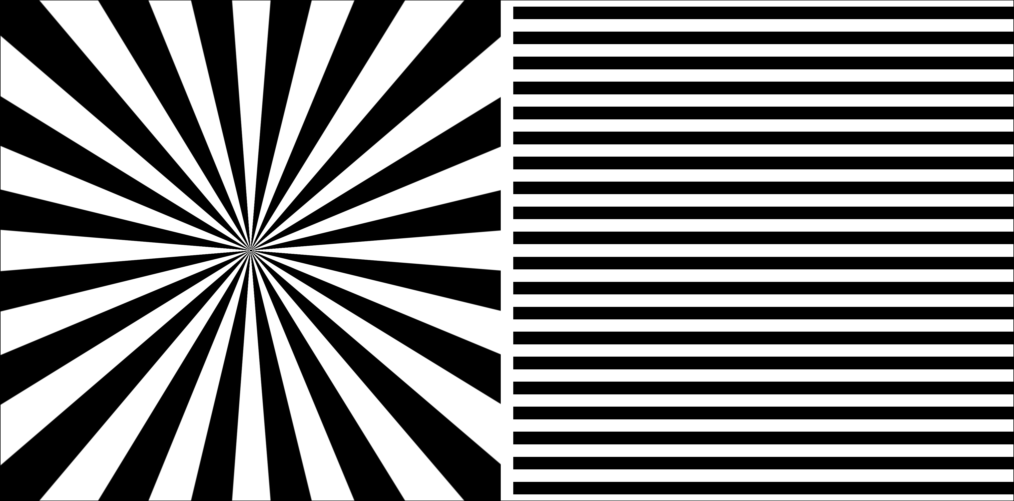}\hspace{0.9em}
	\includegraphics[width=.485\linewidth]{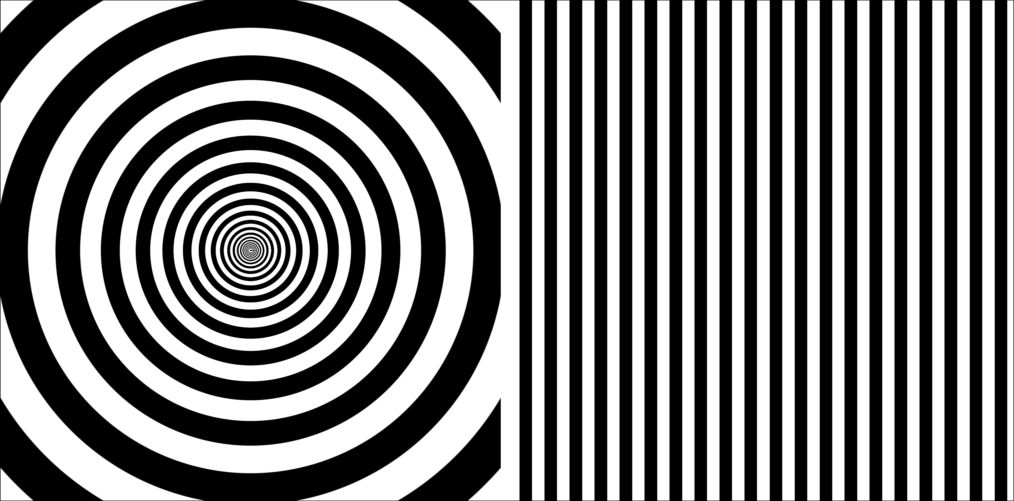}
	\caption{Funnel pattern on the \emph{left} (respectively in the retina and V1). Tunnel pattern on the \emph{right} (respectively in the retina and V1).}
	\label{fig::FT pattern}
\end{figure}
\section{Well-posedness of the Cauchy problem and stationary state}\label{ss::WP}
We recall in this section some preliminary results about the well-posedness of equation \eqref{eq::NF-intro}. We start this section by introducing the definition of stationary state to Equation~\eqref{eq::NF-intro}.
\begin{definition}[Stationary state]\label{def::stationary state to WC equation}
	Let $a_0\in L^p(\R^2)$. For every $I\in L^p(\R^2)$, a stationary state $a_I\in L^p(\R^2)$ to Equation~\eqref{eq::NF-intro} is  a time-invariant solution, viz.
	\begin{equation}\label{eq::SS}\tag{SS}
		a_I = \mu\omega\ast f(a_I)+I.
	\end{equation}
\end{definition}


	
	Using standard assumptions on the kernel $\omega$ or on the response function $f$, it is straightforward to obtain the existence of at least one (even non-constant) stationary state to Equation~\eqref{eq::NF-intro} when $I\equiv 0$, see for instance \cite{bressloff2001,ermentrout1979,nicks2021,curtu2004}. Moreover, in the case of an inhomogeneous equation posed on a bounded domain with a state-dependent sensory input, in \cite{brivadis2022}, under a mild condition on the boundness of the response function, the existence of at least one stationary state is proved using Schaefer’s fixed point Theorem, see also \cite{faugeras2009persistent}. However, in the face of an inhomogeneous equation posed on an unbounded domain as the case at hand, it can become a little bit more subtle to provide the existence of (non-constant) stationary state only with assumptions on $\omega$ and $f$.
	
	Consistent with the strategy that we use in this work, in order to obtain a unique non-constant stationary state to the inhomogeneous Equation~\eqref{eq::NF-intro}, we will make the following assumption on the intra-neuron parameter $\mu>0$,
	\begin{equation}\label{eq::parameter mu_0}
		\mu < \mu_0:= \|\omega\|_1^{-1}.
	\end{equation}
	Observe that $\mu_0\le \mu_c$, where the latter is the bifurcation point defined in \eqref{eq::parameter mu_c}. We stress that, when $p=2$, and under the balanced condition $\widehat{\omega}(0) = 0$, we can relax the above assumption to $\mu<\mu_c$, see Theorem~\ref{thm::existence of stationary input}.
	
	We collect in the following lemma some useful estimates that are immediate consequences of generalised Young-convolution inequality. 
	
	\begin{lemma}\label{Lem::nonlinear operator Q}
		Let $1\le p\le\infty$. The nonlinear operator $X_p\ni a\mapsto \omega\ast f(a)\in X_p$ is well-defined and Lipschitz continuous and
		\begin{equation}\label{eq::nonlinear map Q 1}
			\|\omega\ast f(a)-\omega\ast f(b)\|_{L_x^p L_t^\infty}\le \|\omega\|_1\|a-b\|_{L_x^p L_t^\infty},\qquad\forall a,b\in  X_p.
		\end{equation}
		Moreover,
		\begin{enumerate}
			\item If $a\in X_p$, then $\omega\ast f(a)\in X_{\infty}$ and
			\begin{gather}\label{eq::nonlinear map Q 2}
				\|\omega\ast f(a)\|_{L_x^\infty L_t^\infty}\le \|\omega\|_{q}\|a\|_{L_x^p L_t^\infty},\\
				\label{eq::nonlinear map Q 3}
				\|\omega\ast f(a)\|_{L_x^\infty L_t^\infty}\le \|\omega\|_1;
			\end{gather}
			\item If $a\in X_1$, then $\omega\ast f(a)\in X_p$,
			\begin{equation}\label{eq::nonlinear map Q 4}
				\|\omega\ast f(a)\|_{L_x^p L_t^\infty}\le \|\omega\|_{p}\|a\|_{L_x^1 L_t^\infty}.
			\end{equation}
		\end{enumerate}
	\end{lemma}
	
	In the following theorem, we prove the existence of a unique solution and a unique stationary state of the Cauchy problem associated with Equation~\eqref{eq::NF-intro}.
	\begin{theorem}\label{thm::existence of stationary input}
		Let $1\le p\le\infty$ and $I\in L^p(\R^2)$. For any initial datum $a_0\in L^p(\R^2)$, there exists a unique $a\in X_p$, solution of Equation~\eqref{eq::NF-intro}. 
		Moreover, there exists a unique stationary state $a_I\in L^p(\R^2)$ to \eqref{eq::NF-intro} under the following assumptions:
		\begin{itemize}
			\item[i.] If $\mu<\mu_0$, then
			\begin{equation}\label{eq::Decay of solution}
				\|a(\cdot,t)-a_I(\cdot)\|_{p}\le e^{-(1-\mu \|\omega\|_1)t}\|a_0(\cdot)-a_I(\cdot)\|_p,
				{\qquad \text{for any } t\ge 0.}
			\end{equation}
			\item[ii.] If $p=2$, $\mu<\mu_c$, and $\hat\omega(0)=0$, then
			\begin{equation}\label{eq::Decay of solution in L 2}
				\|a(\cdot,t)-a_I(\cdot)\|_{2}\le e^{-(1-\mu\|\widehat{\omega}\|_\infty)t}\|a_0(\cdot)-a_I(\cdot)\|_2,
				{\qquad \text{for any } t\ge 0.}
			\end{equation}
		\end{itemize}
	\end{theorem}
	\begin{proof}
		Equation~\eqref{eq::NF-intro} can be seen as an ordinary differential equation in $ X_p$, whose r.h.s. is a (globally) Lipschitz continuous map from $ X_p$ to itself by Lemma~\ref{Lem::nonlinear operator Q}. It is then standard to obtain that for any initial datum $a_0\in L^p(\R^2)$, Equation~\eqref{eq::NF-intro} has a unique solution $a\in X_p$ (see, for instance, \cite{veltz2010local,da2011existence}). Moreover, the map $\Phi_I:L^p(\R^2)\to L^p(\R^2)$ defined for all $u\in L^p(\R^2)$ by $\Phi_I(u) = I+\mu\omega\ast f(u)$ satisfies
		$$
		\|\Phi_I(v)-\Phi_I(u)\|_p\le\frac{\mu}{\mu_0}\|v-u\|_p,\qquad \forall u,v\in L^p(\R^2),
		$$
		due to inequality \eqref{eq::nonlinear map Q 1}. Since $\mu<\mu_0$, the existence of a unique stationary state $a_I\in L^p(\R^2)$ is obtained by invoking the contraction mapping principle.
		
		We now present an argument of proof for Item \emph{i.} of the statement.
		Set
		\begin{equation}\label{eq::form of solutions}
			b(x,t) = a(x,t)-a_{I}(x),\qquad (x,t)\in\R ^2\times[0,\infty),
		\end{equation}
		It follows that $b$ is the solution of the following initial value Cauchy problem
		\begin{equation}\label{eq::WC 2}
			\begin{cases}
				\partial_tb(x,t) = -b(x,t)+\mu\displaystyle\int_{\R^2}\omega(x-y)[f(b(y,t)+a_{I}(y))-f(a_{I}(y))]dy,& (x,t)\in\R ^2\times[0,\infty),\cr
				b(x,0) =a_0(x)-a_{I}(x),&x\in\R^2,
			\end{cases}
		\end{equation}
		which belongs to $C([0,\infty);L^p(\R^2))\cap C^1((0,\infty);L^p(\R^2))$. Moreover, $b$ satisfies the following variations of constant formula
		\begin{equation}
			b(x,t)=e^{-t}b(x,0)+ \mu\int_{0}^{t}e^{-(t-s)}\int_{\R^2}\omega(x-y)[f(b(y,t)+a_{I}(y))-f(a_{I}(y))]dy,
		\end{equation}
		for all $(x,t)\in\R ^2\times[0,\infty)$.
		
		Taking the $L^p(\R^2)$-norm of the above identity, we find for every $t\ge 0$,
		\begin{equation}\label{eq::Majoration of the norm of solution}
			\|b(\cdot,t)\|_p\le e^{-t}\|b(\cdot,0)\|_p+\mu \|\omega\|_1\int_{0}^{t}e^{-(t-s)}\|b(\cdot,s)\|_pds.
		\end{equation}
		Applying Gronwall's Lemma to inequality \eqref{eq::Majoration of the norm of solution} one deduces for every $t\ge 0$,
		$$
		\|b(\cdot,t)\|_{p}\le e^{-(1-\mu \|\omega\|_1)t}\|b(\cdot,0)\|_p.
		$$
		This proves the inequality \eqref{eq::Decay of solution} and completes the proof of \emph{i.}.
		
		Let us now prove item \emph{ii.} of the statement. In this case $p=2$, $\mu<\mu_c$, and $\hat\omega(0)=0$.
		The latter condition implies that $\widehat{\omega}(|\xi|)\ge 0$ for all $\xi\in\R^2$. In particular, $\widehat{\omega}(q_c) = \max\limits_{r\ge 0}\widehat{\omega}(r)=\|\widehat{\omega}\|_\infty$. Recall that $\Phi_I(u)=I+\mu\omega\ast f(u)$. Then,  by Plancherel identity, the following holds for all $u,v\in L^2(\R^2)$,
		\begin{eqnarray}\label{eq::Plancherel}
			\|\Phi_I(v)-\Phi_I(u)\|_2= \|\widehat{\Phi_I(v)}-\widehat{\Phi_I(u)}\|_2&=&\mu\|\widehat{\omega}(\widehat{f(u)}-\widehat{f(v)})\|_2\nonumber\\
			&\le&\mu\|\widehat{\omega}\|_\infty\|\widehat{f(u)}-\widehat{f(v)}\|_2\nonumber\\
			&=&\mu\widehat{\omega}(q_c)\|f(u)-f(v)\|_2\nonumber\\
			&\le&\frac{\mu}{\mu_c}\|u-v\|_2.
		\end{eqnarray}
		Here, the last inequality follows from $\mu_c = \hat\omega(q_c)^{-1}$ and the fact that $f$ is $1$-Lipschitz. Since $\mu<\mu_c$, the existence of a unique stationary state $a_I\in L^2(\R^2)$ is obtained by invoking the contraction mapping principle. We complete the proof by arguing as in the previous point and replacing \eqref{eq::Majoration of the norm of solution} by
		\begin{equation}
			\|b(\cdot,t)\|_2\le e^{-t}\|b(\cdot,0)\|_2+\mu \|\hat\omega\|_\infty\int_{0}^{t}e^{-(t-s)}\|b(\cdot,s)\|_2 ds. \qedhere
		\end{equation}
	\end{proof}
	
	Due to Theorem~\ref{thm::existence of stationary input}, we can introduce the following.
	\begin{definition}
		Let $1\le p\le\infty$, the nonlinear input-output map $\Psi:L^p(\mathbb R^2)\to L^p(\mathbb R^2)$ is defined by
		\begin{equation}\label{eq::map Psi}
			\Psi(I) = I+\mu\omega\ast f(\Psi(I)),
			{\qquad\text{for all } I\in L^2(\mathbb{R}^2)}.
		\end{equation}
	\end{definition}
	
	\begin{proposition}
		Let $1\le p\le\infty$ and assume that $\mu<\mu_0$. Then,
		\begin{enumerate}
			\item The map $\Psi$ is well-defined, bi-Lipschitz continuous, and it holds
			\begin{equation}\label{eq::upper bound on Psi_I}
				\|\Psi(I)\|_p\le \frac{\mu_0}{\mu_0-\mu}\|I\|_p,\qquad\mbox{for all}\qquad I\in L^p(\mathbb R^2);
			\end{equation}
			\item If $1< p\le\infty$, the map $\Psi$ belongs to $C^1(L^p(\R^2);L^p(\R^2))$.
		\end{enumerate}
	\end{proposition}
	\begin{proof}
		We only provide the proof of item \emph{1.}, for item \emph{2.} see Theorem~\ref{thm::derivative of Psi}.	Let $I_1, I_2\in L^p(\R^2)$. Then using inequality \eqref{eq::nonlinear map Q 1}, we obtain
		$$
		\|\Psi(I_1)-\Psi(I_2)\|_p \le\frac{\mu}{\mu_0}	\|\Psi(I_1)-\Psi(I_2)\|_p+\|I_1-I_2\|_p.
		$$
		It follows that
		$$
		\|\Psi(I_1)-\Psi(I_2)\|_p\le\frac{\mu_0}{\mu_0-\mu}\|I_1-I_2\|_p,
		$$
		provided $\mu<\mu_0$. This implies that $\Psi$ is Lipschitz continuous from $L^p(\R^2)$ to itself. On the other hand, thanks to inequality \eqref{eq::nonlinear map Q 1},
		\begin{equation}
			\begin{split}
				\|\Psi(I_1)-\Psi(I_2)\|_p 
				&\ge\bigg|\|I_1-I_2\|_p-\mu\|\omega\ast \left[f(\Psi(I_1))- f(\Psi(I_2))\right]\|_p\bigg|\\
				&\ge\|I_1-I_2\|_p-\frac{\mu}{\mu_0}\|\Psi(I_1)-\Psi(I_2)\|_p.
			\end{split}
		\end{equation}
		It follows that
		$$
		\|I_1-I_2\|_p\le\left(1+\frac{\mu}{\mu_0}\right)\|\Psi(I_1)-\Psi(I_2)\|_p.
		$$
		This shows that $\Psi$ is bijective and $\Psi^{-1}$ is Lipschitz continuous from $L^p(\R^2)$ to itself.
	\end{proof}

	\section{Controllability issues of Amari-type equation}\label{s::control Amari equation}
	

	
	In the area of mathematical neuroscience, the research by Ermentrout and Cowan \cite{ermentrout1979} is notable for its pioneering insights into the spontaneous emergence of patterns in V1 using neural fields equations of Wilson-Cowan \cite{wilson1973}. Likewise, Nicks \textit{et al.} \cite{nicks2021} provided a comprehensive understanding of how V1 patterns (orthogonally) respond to specific stimuli that are regular in shape and fill all the visual field, particularly near the threshold value  $\mu_c$ via bifurcation theory and multi-scale analysis. However, these studies do not directly address, for instance, the MacKay effect associated with a funnel pattern that contains highly localized information in the center \cite{mackay1957}  or even Billock and Tsou's experiments \cite{billock2007} since the visual stimuli used in these experiences are non-regular in shape or localized in the visual field.
	
	Building on our discussion in Section~\ref{ss::presentation of results}, we claim that these intriguing visual patterns in V1 should manifest before the $\mu$ parameter reaches the threshold $\mu_c$. Given this, we are interpreting the MacKay effect using a controllability framework, specifically in relation to the Amari-type equation \eqref{eq::NF-intro}. In this context, the sensory input is not just passive data; it acts as a control, shaping and reflecting V1's interpretation of the visual stimulus in the experiment.
	
	We consider in this section the following nonlinear Amari-type control system,
	%
	%
	\begin{equation}\label{eq::nonlinear Amari control system}
		\begin{cases}
			\partial_t a(x,t)+a(x,t)-\mu\displaystyle\int_{\R^2}\omega(x-y)f(a(y,t))dy=I(x)&(x,t)\in\R^2\times[0,T],\cr
			\cr
			a(x,0) = a_0(x),&x\in\R^2,
		\end{cases}
	\end{equation}
	where the cortical activity $a$ represents the state of the system, $a_0\in L^p(\R^2)$ is the initial datum, the sensory input $I\in L^p(\R^2)$ is the control that we will use to act on the system state and the time horizon $T>0$. 
	
	\begin{definition}[Exact controllability]\label{def::Linear control notion}
		Let $1\le p\le\infty$.  We say that the nonlinear control system \eqref{eq::nonlinear Amari control system} is exactly controllable in $L^p(\R^2)$ in time $T>0$ if, for any $a_0, a_T\in L^p(\R^2)$, there exists a control function $I\in L^p(\R^2)$ such that the solution of \eqref{eq::nonlinear Amari control system} with $a(\cdot,0)=a_0$ satisfies $a(\cdot,T) = a_T$.
	\end{definition}

	
	
	To comprehend how the exact controllability of the nonlinear control system \eqref{eq::nonlinear Amari control system} could be handled, let us first investigate the exact controllability of the linear model that we write in a more  abstract way (initial value Cauchy problem) as follows,
	\begin{equation}\label{eq::linear control semigroup notation}
		\dot{a}(t) = Aa(t)+I,\qquad a(0) = a_0,\qquad t\in[0, T],
	\end{equation}
	where the operator $A$ is given by
	\begin{equation}\label{eq::operator A}
		Au = -u+\mu \omega\ast u,\qquad\forall u\in L^p(\R^2).
	\end{equation}
	Observe that for any $1\le p\le\infty$, the operator $A$ is a linear bounded operator from $L^p(\R^2)$ to itself.
	
	\begin{proposition}\label{pro::exact controllability of the linear system}
		Let $1\le p\le\infty$. Then, there exists a positive time $\tau_0>0$ such that the control system \eqref{eq::linear control semigroup notation} is exactly controllable in $L^p(\R^2)$ in any time $\tau \in (0,\tau_0)$.
	\end{proposition}
	
	\begin{proof}
		Fix $T>0$ and let $a_0, a_1\in L^p(\R^2)$. Since $e^{tA}$ is a uniformly continuous semigroup of bounded linear operators on  $L^p(\R^2)$ for any $t\ge 0$, one can write 
		\[
		e^{tA} = \idty+O(t),\qquad t\ge 0,
		\]
		where $O(t)$ is a linear and bounded operator of $L^p(\R^2)$ satisfying $\|O(t)\|\le t\|A\|e^{t\|A\|}$, see for instance, \cite[Theorem~1.2.]{pazy2012}.
		Therefore, for any $\tau\in (0, T]$ and $a_\tau\in L^p(\R^2)$, the solution of \eqref{eq::linear control semigroup notation} at time $\tau$ satisfies
		\begin{equation}
			a(\tau) = e^{\tau A}a_0+\int_{0}^{\tau}e^{(\tau-s)A}Ids = e^{\tau A}a_0+\tau(\idty+O(\tau))I.
		\end{equation}
		Letting $\tau$ small enough, $\idty+O(\tau)$ is invertible in $\mathscr{L}(L^p(\R^2))$ (the vector space of linear and bounded operators from $L^p(\R^2)$ into itself), and $I = \tau^{-1}(\idty+O(\tau))^{-1}(a_1-e^{\tau A}a_0)\in L^p(\R^2)$ defines a control function that steers the solution of  \eqref{eq::linear control semigroup notation} from $a_0$ to $a_1$ in time $\tau$. 
	\end{proof}
	\begin{remark}
		Note that in the linear case, when $\mu<\mu_0$, the control function $I\in L^p(\R^2)$ that steers the solution from the initial state $a_0\in L^p(\R^2)$ to the target state $a_1\in L^p(\R^2)$ in any time $T>0$ is given by
		\begin{equation}
			I = \left(\idty-e^{TA}\right)^{-1}A(a_1-e^{TA}a_0).
		\end{equation}
		Indeed, one can prove that the linear operator $A\in\cL(L^p(\R^2))$ is dissipative when $\mu<\mu_0$ and therefore that  
		\begin{equation}
			\|e^{tA}\|_{\cL(L^p(\R^2))}<1,\qquad \forall t>0.
		\end{equation}
	\end{remark}
	
	Let us now discuss the exact controllability in $L^p(\R^2) $ of the nonlinear system \eqref{eq::nonlinear Amari control system} that we write in abstract way as
	\begin{equation}\label{eq::nonlinear control semigroup notation}
		\dot{a}(t) = N(a(t))+I,\qquad a(0) = a_0,\qquad t\in[0, T],
	\end{equation}
	where for any $1\le p\le\infty$, the nonlinear operator $N$ is defined by
	\begin{equation}\label{eq::operator N}
		N(u) = -u+\mu\omega\ast f(u),\qquad u\in L^p(\R^2).
	\end{equation}
	Then \eqref{eq::nonlinear control semigroup notation} defines an ordinary differential equation in $L^p(\R^2) $ associated with $N$. Recall from Lemma~\ref{Lem::nonlinear operator Q} that for every $1\le p\le\infty$, the nonlinear operator $N$ is (globally) Lipschitz continuous from $L^p(\R^2) $ to itself. Therefore, one can define a  nonlinear\footnote{Please, refer, for instance, to \cite[p. 1]{oharu1966note} or \cite[Section~3]{kato1967nonlinear} for the definition of a nonlinear semigroup of operators.}  semigroup of operators $\{U(t,\cdot)\}_{t\ge 0}$ in $L^p(\R^2)$ such that for any $a_0\in L^p(\R^2)$, $a(t):=U(t, a_0)$ is the unique solution to \eqref{eq::nonlinear control semigroup notation} when $I\equiv0$, namely
	\begin{equation}\label{eq::nonlinear semigroup}
		\displaystyle\frac{\partial U}{\partial t}(t,a_0) = N(U(t,a_0)),\qquad
		U(0, a_0) = a_0.
	\end{equation}
	\begin{theorem}\label{thm::exact controllability of the nonlinear system}
	Let $1< p\le\infty$. Then, there exists a positive time $\tau_0>0$ such that the control system \eqref{eq::nonlinear control semigroup notation} is exactly controllable in $L^p(\R^2)$ in any time $\tau \in (0,\tau_0)$.
	\end{theorem}
	\begin{proof}
		First of all, by Lemma~\ref{lem:differential of map Phi} one has $N\in C^1(L^p(\R^2); L^p(\R^2))$ and the Fréchet differential $DN(u)\in\mathscr{L}(L^p(\R^2))$ is uniformly bounded for every $u\in L^p(\R^2)$. Let $T>0$ and $\tau\in(0,T]$.
		
			By Lemma~\ref{lem::differentiel-U}, the differential of $U$ with respect to $a_0$ is a well-defined invertible operator $D_{a_0}(t,v)\in \mathscr{L}(L^p(\R^2))$ for any $v\in L^p(\R^2)$ and every $0<t\le T$, and it satisfies
			\begin{equation}\label{eq::norm right expo chrono}
				\|D_{a_0}U(t, v)-\idty\|_{\mathscr{L}(L^p(\R^2))}\le t\left(1+\frac{\mu}{\mu_0}\right)e^{\left(1+\frac{\mu}{\mu_0}\right)t},\qquad 0\le t\le T.    
			\end{equation}
			On the other hand, one observes that for any $a_0\in L^p(\R^2)$ and every $I\in L^p(\R^2)$, the solution $a\in C([0, \tau]; L^p(\R^2))$ of \eqref{eq::nonlinear control semigroup notation} can be represented as
			\begin{equation}
				a(t) = U(t,g(t)),
				\qquad\text{where}\quad
				g(t)=a_0+\int_0^t \cT(s)Ids, 
			\end{equation}
			and we let $\cT(s) :=\left[D_{a_0}U(s, g(s))\right]^{-1}$. Again, by Lemma~\ref{lem::differentiel-U}, we have
			\begin{equation}\label{eq::norm left expo chrono}
				\|\cT(s)-\idty\|_{\mathscr{L}(L^p(\R^2))}\le s\left(1+\frac{\mu}{\mu_0}\right)e^{\left(1+\frac{\mu}{\mu_0}\right)s},\qquad 0\le s\le t.
			\end{equation}

		Letting now $a_\tau\in L^p(\R^2)$, one can compute owing to \eqref{eq::norm right expo chrono} and \eqref{eq::norm left expo chrono},
		\begin{eqnarray}
			a_\tau:=a(\tau) = U(\tau, g(\tau))&=&U(\tau, a_0)+\int_0^1 \frac{\partial}{\partial\eta}U\left(\tau, a_0+\eta\int_0^\tau \cT(s)Ids\right)d\eta\nonumber\\
			&=&U(\tau, a_0)+\left\{\int_0^1 \left[D_{a_0}U\left(\tau, a_0+\eta\int_0^\tau \cT(s)Ids\right)\right]\left(\int_0^\tau \cT(s)Ids\right)d\eta\right\}\nonumber\\
			&=&U(\tau, a_0)+\left\{\int_0^1(\idty+O(\tau))\int_0^\tau (\idty+O(s))dsd\eta\right\}I\nonumber\\
			&=&U(\tau, a_0)+\tau(\idty+O(\tau))I.
		\end{eqnarray}
		Then, letting $\tau\in(0, T]$ small enough, one finds that $\idty+O(\tau)$ is invertible in $\mathscr{L}(L^p(\R^2))$ and $I = \tau^{-1}(\idty+O(\tau))^{-1}(a_\tau-U(\tau, a_0))\in L^p(\R^2)$ defines a control function that steers the solution of  \eqref{eq::nonlinear control semigroup notation} from $a_0$ to $a_\tau$ in time $\tau$. 
	\end{proof}
	
	\begin{remark}
	 It is immediate to see that for large time $T>0$, the control $I_1$ defined by $I_1 = 0$ on $[0, T-\tau]$ and $I_1 = \tau^{-1}(\idty+O(\tau))^{-1}(a_1-e^{\tau A}a(T-\tau))$  on $(T-\tau, T]$ (resp. $I_2 = 0$ on $[0, T-\tau]$ and $I_2 = \tau^{-1}(\idty+O(\tau))^{-1}(a_1-U(T,a(T-\tau)))$  on $(T-\tau, T]$)  is a piece-wise constant function in $L^p(\R^2)$ that steers the solution of \eqref{eq::linear control semigroup notation} (resp.  \eqref{eq::nonlinear control semigroup notation}) from $a_0$ to $a_1$ in time $T$.
	\end{remark}
	
	\begin{remark}
Related to Theorem~\ref{thm::exact controllability of the nonlinear system} in the case of $p=1$, the differential $D_{a_0}U(t,v)$ does not belong to $\mathscr{L}(L^1(\R^2))$ since the Gateaux-differential of the nonlinear operator $N$ is not continuous from $L^1(\R^2)$ into $\mathscr{L}(L^1(\R^2))$.

We also stress that exact controllability results proved in Proposition~\ref{pro::exact controllability of the linear system} and Theorem~\ref{thm::exact controllability of the nonlinear system} are quite general regarding assumptions on the connectivity kernel $\omega$ and the nonlinear response function $f$. For the validity of Proposition~\ref{pro::exact controllability of the linear system}, it suffices that $\omega\in L^1(\R^2)$ to guarantee the existence of a unique solution to Equation~\eqref{eq::operator A}. While for the validity of Theorem~\ref{thm::exact controllability of the nonlinear system}, it suffices that $\omega\in L^1(\R^2)\cap L^p(\R^2)$ and $f\in C^2(\R)$ is bounded with its first and second derivatives, as it is sufficient for the validity of Lemma~\eqref{lem:differential of map Phi} (see, for instance, inequalities~\eqref{eq::validity estimate infty} and \eqref{eq::validity estimate}).
	\end{remark}
	
	
\section{On the visual MacKay effect modelling}\label{s::MacKay mathematical interpretation}
As mentioned in the introduction, we stress that the physical visual stimuli employed in MacKay's experiments consist of funnel and tunnel patterns with highly localized information. Taking into account Equation \eqref{eq::funnel and tunnel patterns} and the retino-cortical map, we incorporate these patterns as sensory inputs in Equation \eqref{eq::NF-intro}, such that $I \in \{P_F, P_T\} + \varepsilon v$, where $\varepsilon > 0$ and $v$ represents a localized function in the cortical domain intended to model the highly localized information present in the funnel and tunnel patterns. In this context, the function $v$ can also be regarded as a localized distributed control, aiming to disrupt the \textit{global} plane Euclidean symmetry of the funnel or tunnel pattern.

In Section~\ref{ss::a priori analysis}, assuming that the response function $f$ is linear, we provide a more general result showing that spontaneous cortical patterns cannot induce illusory contours in the output pattern. In particular, we deduce that $I\in \{P_F, P_T\}$  cannot induce the MacKay effect using Equation~\eqref{eq::NF-intro}. Then, using MacKay's stimuli $I\in\{P_F, P_T\}+\varepsilon v$, we prove in Section~\ref{ss::theoretical MacKay effect} that the linearized form of \eqref{eq::NF-intro} is sufficient to replicate the visual MacKay effect theoretically. Thus, the phenomenon starts in the linear regime, so the effect of saturating $f$ should only dampen out high oscillations in the system. Section~\ref{ss::theoretical MacKay effect nonlinear} provides theoretical proof of all these results when the response function $f$ is a nonlinear sigmoid function.

\subsection{A priori analysis}\label{ss::a priori analysis}

In this section, we prove that it is necessary to break the Euclidean symmetry of funnel and tunnel pattern by localized control function for modelling the visual MacKay effect with Equation~\eqref{eq::NF-intro}, both with a linear and nonlinear response function. Our first result is the following.
\begin{theorem}\label{thm::stationary input in linear regime}
	Let $a_0\in L^\infty(\R^2)$ and $I\in L^\infty(\R^2)$ given by $I(\cdot) = \cos(2\pi\langle \xi_0,\cdot\rangle)$, for some $\xi_0\in\R^2$. Assume that the response function $f$ is linear. If $\mu<\mu_0$, it holds
	\begin{equation}
		a(\cdot,t)\xrightarrow[t\to\infty]{} \frac{I(\cdot)}{1-\mu\widehat{\omega}(\xi_0)},\qquad\mbox{exponentially in}\qquad L^\infty(\R^2),
	\end{equation}
	where $a\in X_\infty$ is the solution of \eqref{eq::NF-intro} with initial datum $a_0$.
\end{theorem}
\begin{proof} The stationary state associated with $I(\cdot)=\cos(2\pi\langle \xi_0,\cdot\rangle)$ is given by $a_I(\cdot)=I(\cdot)/(1-\mu\widehat{\omega}(\xi_0))$. Indeed, one has for $x\in\R^2$,
	\begin{equation}
		I(x)+\mu(\omega\ast a_I)(x)=I(x)+\frac{\mu}{1-\mu\widehat{\omega}(\xi_0)}(\omega\ast I)(x)=\frac{I(x)}{1-\mu\widehat{\omega}(\xi_0)} = a_I(x),
	\end{equation}
	since $\omega\ast I = \widehat{\omega}(\xi_0)I$. Therefore, if $\mu<\mu_0$, the result follows by the uniqueness of stationary state and exponential convergence of $a$ to $a_I$ in the space $L^\infty(\R^2)$ provided by Theorem~\ref{thm::existence of stationary input}.
\end{proof}
\begin{corollary}\label{cor::no MacKay effect in the linear regime}
	Assume that the response function $f$ is linear. If $\mu<\mu_0$, then $a_F$ (resp. $a_P$) is a funnel (resp. tunnel) pattern in shape as $P_F$ (resp. $P_T$). In particular, Equation~\eqref{eq::NF-intro} with a linear response function cannot reproduce the MacKay effect starting with a sensory input equal to $P_F$ or $P_T$.
\end{corollary}
\begin{proof}  Due to Theorem~\ref{thm::stationary input in linear regime}, the stationary states associated with $P_F$ and $P_T$ are respectively proportional to $P_F$ and $P_T$ so that they have the same binary pattern respectively (see, Section~\ref{ss:: Binary pattern}), and then the same geometrical shape in terms of images.
\end{proof}

We provide a similar result as that of Theorem~\ref{thm::stationary input in linear regime} with the nonlinear function $f$ in Equation~\eqref{eq::SS}. This result shows, in particular, that even in the presence of the nonlinearity, Equation~\eqref{eq::NF-intro} cannot describe the MacKay effect when the sensory input is chosen equal to $P_F$ or $P_T$. Given that $P_F$ and $P_T$ have symmetrical roles, we focus only on $P_F$. We recall that they are analytically given in Cartesian cortical coordinates in V1 by \eqref{eq::funnel and tunnel patterns}.
\begin{theorem}\label{thm::characterization of properties on a_F}
	Assume that the sensory input in Equation~\eqref{eq::SS}
	is taken as $I = P_F\in L^\infty(\R^2)$. If $\mu<\mu_0$, then the stationary state $a_F:=\Psi(P_F)\in L^\infty(\R^2)$ associated with $P_F$ explicitly depends solely upon $x_2$. Moreover, one has the following.
	\begin{enumerate}
		\item The function $a_F$ is even and $1/\lambda$-periodic with respect to $x_2$;
		\item The function $a_F$ is infinitely differentiable, and Lipshitz continuous;
		\item If in addition $\mu<\mu_0/2$ and the function $f$ is odd, then $a_F$ has a discrete and countable number of zeroes with respect to $x_2$, identical with that of $x_2\mapsto\cos(2\pi\lambda x_2)$ on $\R$.
	\end{enumerate}
\end{theorem}
\begin{remark}\label{rmk::gap between zeroes}
	Notice the assumption $\mu<\mu_0/2$ in item~\emph{3.} of Theorem~\ref{thm::characterization of properties on a_F} instead of $\mu<\mu_0$. We think this is a technical assumption because of the strategy used in our proof since numerical results suggest that item~\emph{3.} remains valid for all $\mu<\mu_0$.  Moreover, the assumption on the parity of $f$ is also technical, and we conjecture that if $f$ is not odd, then $a_F$ will still have a discrete and countable number of zeroes with respect to $x_2$, such that
	\begin{equation}\label{eq::zeroes of a_F}
		a_F^{-1}(\{0\}) = \R\times\left\{z_k\in\left]\frac{k}{2\lambda},\frac{k+1}{2\lambda}\right[\mid k\in\Z\right\},
	\end{equation}
	and, for all $k\in\Z$,
	\begin{equation}\label{eq::gap between z-k and tau_k}
		|z_k-\tau_k|\le\frac{\arcsin(\mu\mu_0^{-1})}{2\pi\lambda},\qquad\mbox{where}\qquad \tau_k := \frac{2k+1}{4\lambda}.
	\end{equation}
	The gap between the zeroes of $a_F$ and those of $P_F$ provided by \eqref{eq::gap between z-k and tau_k} shows that on each interval, $z_k$ and $\tau_k$ become arbitrarily close depending on whether $ \mu$ is not closed to $\mu_0$. Nevertheless, if $\lambda$ is taken sufficiently large, $z_k$ and $\tau_k$ become arbitrarily close independently of $\mu_0-\mu$.
\end{remark}

\begin{proof}[Proof of Theorem~\ref{thm::characterization of properties on a_F}]
	We assume that $\lambda=1$ in the sequel for ease of notation.
	We know by item \emph{3.} of Proposition~\ref{pro::equivariance of the convolution operator} that if $P_F$ or $a_F$ has a subgroup of $\mathbf{E}(2)$ as a group of symmetry, the other has the same subgroup as a group of symmetry and conversely. Since $P_F(x_1,x_2)$ is independent on $x_1$, it follows that $a_F(x_1,x_2)$ is also independent on $x_1$ for all $(x_1,x_2)\in\R^2$. Similarly, since $P_F$ is invariant under the action of the reflection with respect to the straight $x_1 = 0$, that is, $P_F(x_1,-x_2) = P_F(x_1,x_2)$ for all $(x_1,x_2)\in\R^2$, one deduces that $a_F(x_1,-x_2) = a_F(x_1,x_2)$. Thus, $a_F$ is an even function with respect to $x_2$. Similarly, since $P_F$ is invariant under the translation by vector $(0,-1)\in\R^2$, it follows that $a_F$ is also invariant under this translation so that $a_F$ is $1$-periodic with respect to $x_2$. The fact that $a_F$ is infinitely differentiable on $\R^2$ follows immediately from that $P_F\in C^\infty(\R^2)$, the kernel $\omega\in\cS(\R^2)\subset C^\infty(\R^2)\cap L^1(\R^2)$ and that $f$ is bounded. Writing now $a_F(x_2):=a_F(x_1,x_2)$ for notational ease, we obtain that $a_F$ is also given by
	\begin{equation}\label{eq::nonlinear 1D ss}
		a_F(x_2) = \cos(2\pi x_2)+\mu[\omega_1\ast f(a_F)](x_2),\qquad x_2\in\R,
	\end{equation}
	where $\omega_1$  is a $1$D difference of Gaussian kernel . Let $a_F':=\partial_{x_2}a_F$, then due to \eqref{eq::nonlinear 1D ss}, one obtains
	\begin{equation}\label{eq::nonlinear 1D ss for derivative}
		a_F'(x_2) = -2\pi\sin(2\pi x_2)+\mu[\omega_1'\ast f(a_F)](x_2),\qquad x_2\in\R.
	\end{equation}
	Since $\|f\|_\infty\le 1$ by assumption, it follows that $\|a_F'\|_\infty\le 2\pi+\mu\|\omega_1'\|_1<\infty$. Therefore, $a_F$ is Lipschitz continuous. 
	
	We now present an argument to prove item~\emph{3.} of  Theorem~\ref{thm::characterization of properties on a_F}. Notice that $a_F = \Psi(P_F)$ satisfies
	\begin{equation}
		\label{eq:AT}
		a_F = P_F + \mu\omega\ast f(a_F).
	\end{equation}
	Let $x^{*}:=(x_1^{*},x_2^{*})\in\R^2$ be such that $P_F(x^{*}) = 0$. It follows from \eqref{eq:AT} that
	\begin{equation}\label{eq::zero Psi intermediate}
		a_F(x^{*})  = \mu\int_{\R^2}\omega(y)f(a_F(x^{*}-y))dy.
	\end{equation}
	By using \eqref{eq:AT} once again, one obtains 
	\begin{equation}\label{eq::Psi intermediate}
		a_F(x^{*}-y) = \cos(2\pi(x_2^{*}-y_2))+\mu\int_{\R^2}\omega(y-z)f(a_F(x^{*}-z))dz.
	\end{equation}
	We introduce the map $\displaystyle g_3:y:=(y_1,y_2)\in\R^2\longmapsto g_3(y) := a_F(x^{*}-y)$. Then, it is straightforward to observe that, $g_3$ is the unique solution of \eqref{eq::Psi intermediate} for every $\mu<\mu_0$, and it holds
	$$
	-a_F(x^{*}+y) = \cos(2\pi(x_2^*-y_2))+\mu\int_{\R^2}\omega(y-z)f(-a_F(x^{*}+z))dz,
	$$
	since $f$ is odd. So the function $y\in\R^2\mapsto -g_3(-y)$ is also solution of \eqref{eq::Psi intermediate}. By uniqueness of solution, one has $g_3(-y) = -g_3(y)$ and that $y\in\R^2\longmapsto\omega(y)f(a_F(x^{*}-y))$ is an odd function, since $\omega$ is symmetric and $f$ is and odd function. It follows that the right-hand side of \eqref{eq::zero Psi intermediate} is equal to $0$.
	
	To show the converse inclusion, let $x^{*}:=(x_1^*,x_2^*)\in\R^2$ verifying $a_F(x^{*}) = 0$. From \eqref{eq:AT}, it follows
	\begin{equation}\label{eq::useful-1}
		\cos(2\pi x_2^*) = -\mu\int_{\R^2}\omega(y)f(a_F(x^{*}-y))dy.
	\end{equation}
	Developing $\cos(2\pi(x_2^{*}-y_2))$ and replacing in \eqref{eq::Psi intermediate} $\cos(2\pi x_2^*)$ by its expression in \eqref{eq::useful-1}, one obtains for $y\in\R^2$,
	\begin{equation}\label{eq:x-y}
		a_F(x^{*}-y) =\sin(2\pi x_2^*)\sin(2\pi y_2)
		+\mu\int_{\R^2}k(y,z)f(a_F(x^{*}-z))dz,
	\end{equation}
	where $k(y,z) := \omega(y-z)-\cos(2\pi y_2)\omega(z),$ satisfies
	\begin{equation}
		K:=\sup\limits_{y\in\R^2}\int_{\R^2}|k(y,z)|dy\le 2\|\omega\|_1.
	\end{equation}
	Since $\mu<\mu_0/2$, the contracting mapping principle shows that for every $I\in L^\infty(\R^2)$ there exists a unique solution $b\in L^\infty(\R^2)$ to
	\begin{equation}\label{eq::fixed point 1}
		b(y) = I(y)+\mu\int_{\R^2}k(y,z)f(b(z))dz.
	\end{equation}
	By \eqref{eq:x-y}, function $b(y):=a_F(x^{*}-y)$ is the unique solution  of the above equation associated with $I(y)=\sin(2\pi x_2^*)\sin(2\pi y_2)$.
	
	On the other hand, since $\omega$ is symmetric and the sigmoid $f$ is an odd function, we have also for $\operatorname{a.e.}, y\in\R^2$,
	\begin{equation}
		-a_F(x^{*}+y) =
		\sin(2\pi x_2^*)\sin(2\pi y_2)
		+\mu\int_{\R^2}k(y,z)f(-a_F(x^{*}+z))dz,
	\end{equation}
	so that, the function $\tilde b(y) = -b(-y)$ is also solution of Equation~\eqref{eq::fixed point 1} associated with the input $I(y)=\sin(2\pi x_2^*)\sin(2\pi y_2)$. By uniqueness of solution, one then has $b(-y) = -b(y)$ for $\operatorname{a.e.}, y\in\R^2$. This shows that $y\mapsto\omega(y)f(a_F(x^{*}-y))$ is an odd function on $\R^2$, since $\omega$ is symmetric and $f$ is an odd function, which implies that the r.h.s. of \eqref{eq::useful-1} is equal to $0$ and thus that $x^*\in P_F^{-1}(\{0\})$.
\end{proof}

The proof of the following corollary follows the same lines as that of Corollary~\ref{cor::no MacKay effect in the linear regime}.
\begin{corollary}\label{cor::no MacKay effect in the nonlinear regime}
	Under assumption, $\mu<\mu_0$, $a_F$ (resp. $a_T$) is a funnel (resp. tunnel) pattern in shape. In particular, Equation~\eqref{eq::NF-intro} with a sigmoid activation function cannot reproduce the MacKay effect starting with a sensory input equal to $P_F$ or $P_T$.
\end{corollary}

\begin{remark}
	By following the lines in the proof of Theorem~\ref{thm::characterization of properties on a_F}, we can notice that it is only sufficient for the kernel $\omega$ to be homogeneous and isotropically invariant to obtain the desired results.
\end{remark}
	
\subsection{The visual MacKay effect with a linear response function}\label{ss::theoretical MacKay effect}

The results we provide in this section aim to replicate the MacKay effect using Equation~\eqref{eq::NF-intro} when the response function $f$ is linear. The Corollary~\ref{cor::no MacKay effect in the linear regime} shows that, for our model of cortical activity in V1, one cannot obtain the MacKay effect in the linear regime without breaking the Euclidean plane symmetry of the sensory input when chosen equal to $P_F$ or $P_T$. Our purpose now is to show that Equation~\eqref{eq::NF-intro} with the linear response function and sensory input $I\in \{P_F, P_T\}+\varepsilon v$ reproduces the MacKay effect. Here, $v$ is a suitable control function that should model the localized information in MacKay's stimuli.
\begin{remark}
	We notice that only the description of the MacKay effect related to the funnel pattern will be shown for ease of presentation and reader convenience.  Then, in the rest of this section, we focus on describing the MacKay effect related to the ``MacKay rays''; see Fig~\ref{fig:mackay}.
\end{remark}

{One of the essential characteristics of the retino-cortical map, i.e. the way the visual field is projected into V1, is that small objects located in the fovea, which is the center of the visual field, have a much more extensive representation in V1 than similar objects located in the peripheral visual field. As a result, for the cortical representation of the ``MacKay rays'' visual stimulus, we choose a sensory input in Equation~\eqref{eq::NF-intro} as $I(x) = P_F(x)+\varepsilon H(\theta-x_1)$, where $\varepsilon>0$, $\theta\in\R$ and $H$ is the Heaviside step function. This choice models the highly localized information in the center of the funnel pattern \textit{created by the very fast alternation of black and white rays}. It is worth noting that this corresponds to localized information in horizontal stripes in the left area of the cortex}.


To keep the presentation as clear as possible for reader convenience, we let $\theta = 0$, and we assume that the cortical representation of the ``MacKay rays'' visual stimulus is given by
\begin{equation}\label{eq::input for MacKay rays}
	I(x) = \cos(2\pi\lambda x_2)+\varepsilon H(-x_1),\qquad\lambda,\;\varepsilon>0,\; x:=(x_1,x_2)\in\R^2.
\end{equation}


The sensory input $v(x_1,x_2) = H(-x_1)$ is dependent only on the variable $x_1$. As a result of Remark~\ref{rmk::equivariance of the convolution operator}, the associated stationary output $b$ is also dependent solely on that variable. Therefore, our current task is to compute the solution $b$ of the following equation
\begin{equation}\label{eq::SS dim 1}
	b(x) = I(x)+\mu(\omega_1\ast b)(x),\qquad\qquad x\in\R,
\end{equation}
where $I(x) = H(-x)$ and the $1$-D kernel $\omega_1$ is given by
\begin{equation}\label{eq::connectivity 1-D}
	\omega_1(x) = [\sigma_1\sqrt{2\pi}]^{-1}e^{-\frac{x^2}{2\sigma_1^2}}-\kappa[\sigma_2\sqrt{2\pi}]^{-1}e^{-\frac{x^2}{2\sigma_2^2}},\qquad x\in\R.
\end{equation}
\begin{equation}\label{eq::fourier transform in 1d}
	\widehat{\omega_1}(\xi)=e^{-2\pi^2\sigma_1^2\xi^2}-\kappa
	e^{-2\pi^2\sigma_2^2\xi^2}.
\end{equation}
\begin{lemma}\label{lem::important}
	Let $I\in\cS'(\R)$ and the kernel $\omega_1\in\cS(\R)$ be defined by \eqref{eq::connectivity 1-D}. Under the assumption $\mu<\mu_c$, there is a unique solution $b\in\cS'(\R)$ to Equation~\eqref{eq::SS dim 1}, which is given by
	\begin{equation}\label{eq::solution of SS DoG}
		b = I+\mu K\ast I.
	\end{equation}
	Here the kernel $K\in\cS(\R)$ is defined of all $x\in\R$ by
	\begin{equation}\label{eq:: K}
		K(x) = \int_{-\infty}^{+\infty} e^{2i\pi\xi x}\widehat{K}(\xi) d\xi,\qquad\mbox{where}\qquad \widehat{K}(\xi) = \frac{\widehat{\omega_1}(\xi)}{1-\mu\widehat{\omega_1}(\xi)},\qquad\qquad\forall \xi\in\R.
	\end{equation}
\end{lemma}
\begin{proof}
	First of all, under assumptions on $I$ and $\omega_1$, we have that Equation~\eqref{eq::SS dim 1} is well-posed in $\cS'(\R)$. Then taking respectively the Fourier transform of \eqref{eq::SS dim 1} and the inverse Fourier transform in the space $\cS'(\R)$, we find that $b\in\cS'(\R)$ is given by \eqref{eq::solution of SS DoG} with $K\in\cS(\R)$ defined as in \eqref{eq:: K}. Indeed, observe that $\widehat{K}$ is well-defined on $\R$ due to hypothesis $\mu<\mu_c$, with $\mu_c$ being defined in \eqref{eq::parameter mu_c}, and it belongs to the Schwartz space $\cS(\R)$ as the product of a $C^\infty(\R)$ function and an element of $\cS(\R)$.
\end{proof}

Due to Lemma~\ref{lem::important}, inverting the kernel $K$ defined in \eqref{eq:: K} and providing an asymptotic behaviour of its zeroes on $\R$ will help to provide detailed information on the qualitative properties of the function $b$ as given by \eqref{eq::solution of SS DoG}. To achieve this, we use tools from complex. 

Let us consider the extension of $\widehat{K}$ in the set $\C$ of complex numbers,
\begin{equation}
	\widehat{K}(z) = \frac{\widehat{\omega_1}(z)}{1-\mu\widehat{\omega_1}(z)},\qquad z\in\C.
\end{equation}
Then $\widehat{K}$ is a meromorphic function on $\C$, and its poles are zeroes of the entire function
\begin{equation}\label{eq::exponential polynomial h}
	h(z):= 1-\mu e^{-2\pi^2\sigma_1^2z^2}+\kappa\mu e^{-2\pi^2\sigma_2^2z^2},\qquad z\in\C.
\end{equation}
\begin{remark}\label{rmk::exponential polynomial}
	The holomorphic function $h$ is an exponential polynomial \cite[Chapter 3]{berenstein2012} in $-z^2$ with frequencies $\alpha_0 = 0$, $\alpha_1 = 2\pi^2\sigma_1^2$ and $\alpha_2 = 2\pi^2\sigma_2^2$ satisfying $\alpha_0<\alpha_1<\alpha_2$ due to assumptions on $\sigma_1$ and $\sigma_2$. It is \textit{normalized} since the coefficient of $0$-frequency equals $1$. A necessary condition for $h$ for being \textit{factorizable} \cite[Remark~3.1.5, p.~201]{berenstein2012} is that parameters $\sigma_1$ and $\sigma_2$ are taken so that it is \textit{simple}. By definition \cite[Definition 3.1.4, p. 201]{berenstein2012}, $h$ is simple if $\alpha_1$ and $\alpha_2$ are commensurable, i.e., $\alpha_1/\alpha_2\in\Q$, which is equivalent to $\sigma_1^2/\sigma_2^2\in \Q$. Here $\Q$ denote the set of rational numbers.
\end{remark}
\begin{remark}\label{rmk::particular considerations}
	For ease in computation and pedagogical presentation, we assume in the rest of this section that parameters in the kernel $\omega$ defined in \eqref{eq::connectivity} and then in the $1$-D kernel $\omega_1$ defined in \eqref{eq::connectivity 1-D} are given by $\kappa=1$, $2\pi^2\sigma_1^2=1$ and $2\pi^2\sigma_2^2=2$. In this case, one has $\mu_0 := \|\omega\|_1^{-1} = 2$ and $\mu_c:= \widehat{\omega}(\xi_c)^{-1} = 4$. In particular $\omega$ satisfies the balanced condition $\widehat{\omega}(0) =0$. Then, assuming in this particular consideration that $\mu:=1<\mu_0=2$ is not a loss of generality.
\end{remark}

The main result of this section is then the following.
\begin{theorem}\label{thm::stationary state for MacKay rays}
	Assume that the response function $f$ is linear and the input $I$ is given by \eqref{eq::input for MacKay rays}. Under the considerations of Remark~\ref{rmk::particular considerations}, the unique stationary state to Equation~\eqref{eq::NF-intro} is given for all $(x_1,x_2)\in\R^2$ by
	\begin{equation}\label{eq::stationary state for MacKay rays}
		a_I(x_1,x_2) = \frac{\cos(2\pi\lambda x_2)}{1-\mu\widehat{\omega}(\xi_0)}+\varepsilon g(x_1),\qquad\qquad\xi_0:=(0,\lambda),
	\end{equation}
	where $g:\R\to\R$ has a discrete and countable set of zeroes on $(0,+\infty)$.
\end{theorem}

Observe that under the assumption that the response function $f$ is linear, Equation~\eqref{eq::NF-intro} becomes linear. It follows that the first term in the r.h.s. of \eqref{eq::stationary state for MacKay rays} is the stationary output associated with the input $P_F$ by using Theorem~\ref{thm::stationary input in linear regime}, and $b$ is the stationary output associated with the sensory input $v(x_1,x_2) = H(-x_1)$.
Consequently, Theorem~\ref{thm::stationary state for MacKay rays} follows from the following proposition. The proof is an immediate consequence of Lemma~\ref{lem::important}, Theorem~\ref{thm::main 2} and Proposition~\ref{prop::zeroes of K} given in Section~\ref{ss::complement MacKay linear regime}.

\begin{proposition}\label{pro::ss for Heaviside}
	Let $I(x) = H(-x)$, $x\in\R$, $H$ being the Heaviside step function. Under the considerations of Remark~\ref{rmk::particular considerations}, the solution $b\in L^\infty(\R)$ of \eqref{eq::SS dim 1} is given, for $x>0$, by
	\begin{equation}\label{eq::sol heaviside}
		e^{\pi x\sqrt{\frac{2\pi}{3}}}b(x) =\frac{\sqrt{3}}{\pi}\cos\left(\frac{\pi}{3}+\pi x\sqrt{\frac{2\pi}{3}}\right)+O\left(\frac{1}{x}\right).
	\end{equation}
	Moreover, letting $(\theta_k)_{k\in\N^{*}}$ and $(\tau_k)_{k\in\N^{*}}$ be respectively zeroes and extrema of $x\mapsto \cos(\pi/3+\pi x\sqrt{2\pi/3})$ for $x>0$, the zeroes of $b$ in $(0,+\infty)$ are a countable sequence $(\rho_k)_{k\in\N^{*}}$  such that $\rho_k$ is unique in the interval $J_k:=]\tau_k,\tau_{k+1}[$ for all $k\in\N^{*}$ and it holds
	\begin{equation}\label{eq::asymptotic of zeroes of b}
		|\theta_{k+1}-\rho_k|\le\frac{\sqrt{6}}{2\pi^2}\arcsin\left(\frac{2\sqrt{5}}{5\pi(3k-1)}\right),\qquad\forall k\in\N^{*}.
	\end{equation}
\end{proposition}
\begin{proof}
	If  $I(x) = H(-x)$, $x\in\R$, is the input in Equation~\eqref{eq::SS dim 1}, then by Lemma~\ref{lem::important} and Theorem~\ref{thm::main 2}, the solution $b\in L^\infty(\R)$ of \eqref{eq::SS dim 1} is given for all $x>0$ by
	\begin{eqnarray}\label{eq::inter_1}
		\frac{b(x)}{2\sqrt{\pi}} &=&\int_{x}^{+\infty}e^{-\pi y\sqrt{\frac{2\pi}{3}}}\cos\left(\frac{\pi}{12}+\pi y\sqrt{\frac{2\pi}{3}}\right)dy\nonumber\\
		&&+ \int_{x}^{+\infty}\sum\limits_{k=1}^{+\infty}\frac{e^{-\pi c_ky\sqrt{\frac{2\pi}{3}}}}{c_k}\cos\left(\frac{\pi}{12}+\pi c_ky\sqrt{\frac{2\pi}{3}}\right)dy\nonumber\\
		&&+\int_{x}^{+\infty}\sum\limits_{k=1}^{+\infty}\frac{e^{-\pi d_ky\sqrt{\frac{2\pi}{3}}}}{d_k}\sin\left(\frac{\pi}{12}-\pi d_ky\sqrt{\frac{2\pi}{3}}\right)dy,
	\end{eqnarray}
	where the sequences $(c_k)_k$ and $(d_k)_k$ are defined in \eqref{eq:: c_k and d_k}. Since these two sequences are positives and tend to $+\infty$ as $k\to+\infty$, we can commute the integrals and the sums in the r.h.s of \eqref{eq::inter_1} for all $x>0$. One finds,
	\begin{eqnarray}
		b(x) &=& \frac{\sqrt{3}}{\pi}\cos\left(\frac{\pi}{3}+\pi x\sqrt{\frac{2\pi}{3}}\right)e^{-\pi x\sqrt{\frac{2\pi}{3}}}+\frac{\sqrt{3}}{\pi}\sum\limits_{k=1}^{+\infty}\frac{e^{-\pi c_k x\sqrt{\frac{2\pi}{3}}}}{c_k^2}\cos\left(\frac{\pi}{3}+\pi c_k x\sqrt{\frac{2\pi}{3}}\right)\nonumber\\
		&&-\frac{\sqrt{3}}{\pi}\sum\limits_{k=1}^{+\infty}\frac{e^{-\pi d_k x\sqrt{\frac{2\pi}{3}}}}{d_k^2}\cos\left(\frac{\pi}{3}+\pi d_k x\sqrt{\frac{2\pi}{3}}\right),
	\end{eqnarray}
	and \eqref{eq::sol heaviside} immediately follows. Finally, to prove \eqref{eq::asymptotic of zeroes of b}, it suffices to repeat step by step the proof of Lemma~\ref{Lem:estim0} and Proposition~\ref{prop::zeroes of K} given in Section~\ref{ss::complement MacKay linear regime}.
\end{proof}

The Proposition~\ref{pro::ss for Heaviside} implies that if the sensory input is the V1 representation of the ``MacKay rays'' as defined by \eqref{eq::input for MacKay rays}, then the associated stationary state corresponds to the V1 representation of the afterimage reported by MacKay \cite{mackay1957}. Moreover, Theorem~\ref{thm::existence of stationary input} ensures that the average membrane potential $a(x,t)$ of neurons in V1 located at $x\in\R^2$ at time $t\ge 0$ exponentially stabilises on the stationary state when $t\to\infty$. It follows that Equation~\eqref{eq::NF-intro} theoretically replicates the MacKay effect associated with the ``MacKay rays'' at the cortical level. Due to the retino-cortical map between the visual field and V1, we deduce the theoretical description of the MacKay effect for the ``MacKay rays'' at the retinal level.
\begin{remark}\label{rmk::comments on the MacKay target modelling}
	We emphasise that a linear combination of the Heaviside step function in the $x_2$-variable as a perturbation of the V1 representation of the tunnel pattern (called ``MacKay target'') gives rise to the MacKay effect description related to this pattern.
	
	{ While in the funnel pattern used in MacKay experiences, the difference between the center and the rest of the image is more evident, one must notice that the MacKay target stimulus (Figure~\ref{fig::MacKay target}, the right panel) is not perfectly radial. Indeed, symmetry-breaking imperfections are present. We chose to mimic these imperfections by shaping our control as concentrated on rays converging to the origin, which we chose to be symmetric for mathematical convenience (this has no bearing on the results). This control captures that a symmetry-breaking in the input is necessary to induce the observed afterimage and is sufficiently simple to be mathematically tractable.  We stress that the effect of adding such rays is barely noticeable (see, e.g., the left panel in Figure~\ref{fig::mackay-target-nonlinear}, where the two rays are along the line passing horizontally through the center of the image). }
\end{remark}

\subsection{The visual MacKay effect with a nonlinear response function}\label{ss::theoretical MacKay effect nonlinear}

This section aims to show that Equation~\eqref{eq::NF-intro} with a nonlinear response function $f$ still replicates the MacKay effect associated with the ``MacKay rays'' and the ``MacKay target'', see Figure~ \ref{fig::MacKay target}.

Remark~\ref{rmk::gap between zeroes} and Corollary~\ref{cor::no MacKay effect in the nonlinear regime} shows that, for our model of cortical activity in V1 modelled by Equation~\eqref{eq::NF-intro}, one cannot replicate the MacKay effect even with a nonlinear response function (having standard properties in most neural fields model, namely, a sigmoid) without breaking the Euclidean plane symmetry of the sensory input when chosen equal to $P_F$ or $P_T$. In the following, in order to see why a response function with sigmoid shape replicates the MacKay effect, we assume the following hypothesis.
\begin{hypothesis}\label{hyp::on the response function}
	The response function $f$ satisfies: $f\in C^2(\R)$, $f$ is odd and $f(s)=s$ for all $|s|\le 1$. We also assume that $\max_{s\in\R} f'(s)=1$.
\end{hypothesis}

Let us model the cortical representation of the ``MacKay rays'' input by the following
\begin{equation}\label{eq::MacKay rays we control}
	I(x) = \gamma P_F(x)+\varepsilon H(-x_1),\qquad x:=(x_1,x_2)\in\R^2,
\end{equation}
where  $\gamma\ge0$ is a control parameter, $\varepsilon>0$ and $P_F(x) = \cos(2\pi\langle\xi_0,x\rangle)$ is an analytical representation of the funnel pattern in cortical coordinates, where $\xi_0=(0,\lambda)$ with $\lambda>0$.

The first result of this section is then the following
\begin{proposition}\label{pro::1-nonlinear MacKay}
	Let the input $I$ be defined by \eqref{eq::MacKay rays we control} with $\varepsilon>0$ small and $\gamma\le 1-\mu\widehat{\omega}(\xi_0)$. Under the assumption, $\mu<\mu_0$, equation \eqref{eq::NF-intro} with a response function satisfying Hypothesis~\ref{hyp::on the response function}  replicates the MacKay effect associated with the ``MacKay rays''.
\end{proposition}
\begin{proof}
	On one hand, the stationary solution associated with $I(x) = \gamma P_F(x)+\varepsilon v(x_1,x_2)$, where $v(x_1,x_2) = H(-x_1)$ satisfies \eqref{eq::map Psi} in $L^\infty(\R^2)$, i.e.,
	\begin{equation}\label{eq::int 1}
		\Psi(\gamma P_F+\varepsilon v) = \gamma P_F+\varepsilon v+\mu\omega\ast f(\Psi(\gamma P_F+\varepsilon v)).
	\end{equation}
	On the other hand, since $\|P_F\|_\infty = 1 = \|v\|_\infty$, $0<\gamma\le 1-\mu\widehat{\omega}(\xi_0)\le 1$ and $\varepsilon\ll 1$, we can apply Theorem~\ref{thm::derivative of Psi} and obtain
	\begin{equation}\label{eq::int 2}
		\Psi(\gamma P_F+\varepsilon v) = \Psi(\gamma P_F)+\varepsilon D\Psi(\gamma P_F)v+o(\varepsilon),
	\end{equation}
	where $D\Psi(\gamma P_F)v$ is the differential of $\Psi$ at $\gamma P_F$ in the direction of $v$. It also follows from Theorems~\ref{thm::existence of stationary input} and \ref{thm::main theorem 3} that for some $g_1\ge 0$, it holds   $\|\Psi(\gamma P_F)\|_\infty\le g_1 = \gamma\|P_F\|_\infty+(\mu/\mu_0)f(g_1)<3/2$. Thus, injecting \eqref{eq::int 2} into \eqref{eq::int 1} and Taylor expansion of $f$ in the first order leads to
	\begin{equation}\label{eq::int 3}
		\Psi(\gamma P_F) = \gamma P_F+\mu\omega\ast f(\Psi(\gamma P_F)),\quad D\Psi(\gamma P_F)v = v+\mu\omega\ast[f'(\Psi(\gamma P_F))D\Psi(P_F)v.]
	\end{equation}
	Thanks to Hypothesis~\ref{hyp::on the response function} and the assumption $\gamma\le 1-\mu\widehat{\omega}(\xi_0)$, one has $\Psi(\gamma P_F) =\gamma P_F/(1-\mu\widehat{\omega}(\xi_0))$. Indeed, since  $|\gamma P_F/(1-\mu\widehat{\omega}(\xi_0)|\le 1$ and $\omega\ast P_F = \widehat{\omega}(\xi_0) P_F$, one has that
	\begin{equation}
		\gamma P_F+\mu\omega\ast f\left(\frac{\gamma P_F}{1-\mu\widehat{\omega}(\xi_0)}\right)=\gamma P_F+\frac{\mu\gamma \omega\ast P_F}{1-\mu\widehat{\omega}(\xi_0)}=\frac{\gamma P_F}{1-\mu\widehat{\omega}(\xi_0)}.
	\end{equation}
	Therefore, $\Psi(\gamma P_F)$ is also a funnel pattern when represented in term of binary image. Moreover, since $|\Psi(\gamma P_F)|\le 1$, one has $f'(\Psi(\gamma P_F))=1$, and $D\Psi(\gamma P_F)v = v+\mu\omega\ast D\Psi(P_F)v$ has a discrete and countable set of zeroes by Proposition~\ref{pro::ss for Heaviside}.  The result then follows at once.
\end{proof}
\begin{proposition}\label{pro::2-nonlinear MacKay}
	Let $v\in L^\infty(\R^2)$. Under the assumption $\mu<\mu_0$, the map $\Pi:\gamma\in\R_{\ge 0}\mapsto\Pi(\gamma) = u_\gamma\in L^\infty(\R^2)$, where $u_\gamma$ is the solution of $u_\gamma=v+\mu\omega\ast \left[f'(\Psi(\gamma P_F))u_\gamma\right]$ is Lipschitz continuous.
\end{proposition}
\begin{proof}
	Let $v\in L^\infty(\R^2)$ be fixed and $\gamma\in\R_{\ge 0}$. If $u_\gamma\in L^\infty(\R^2)$ is the solution of $u_\gamma=v+\mu\omega\ast \left[f'(\Psi(\gamma P_F))u_\gamma\right]$, then, under the assumption $\mu<\mu_0$ and Hypothesis~\ref{hyp::on the response function}, one has $\|u_\gamma\|_\infty\le \|v\|_\infty\mu_0/(\mu-\mu_0)$. Let now $\gamma_1, \gamma_2\in\R_{\ge 0}$, then using Inequality~\eqref{eq::upper bound on Psi_I}, one finds
	\begin{eqnarray}
		\|\Pi(\gamma_1)-\Pi(\gamma_2)\|_\infty&\le&\mu\|\omega\|_1 \|f'(\Psi(\gamma_1 P_F))u_{\gamma_1}-f'(\Psi(\gamma_2 P_F))u_{\gamma_2}\|_\infty\nonumber\\
		&\le&\frac{\mu}{\mu_0}\|\Pi(\gamma_1)-\Pi(\gamma_2)\|_\infty+\frac{\mu\mu_0\|v\|_\infty f_\infty''}{(\mu_0-\mu)^2}|\gamma_1-\gamma_2|,
	\end{eqnarray}
	where $f_\infty''$ is the $L^\infty$-norm of the second derivative $f''$. The result then follows at once.
\end{proof}

Let us define the positive quantity
\begin{equation}\label{eq:: gamma 0}
	\gamma_0 := \sup\{\gamma\ge 0\mid \|\Psi(\gamma' P_F)\|_\infty\le 1, \mbox{for all } \gamma'\in[0,\gamma]\}.
\end{equation}
Observe that $\gamma_0$ is not necessary finite and that if $0\le\gamma\le\gamma_0$, then $f'(\Psi(\gamma P_F)) = 1$. It follows that if $\gamma_0 =+\infty$, then $\|\Psi(\gamma P_F)\|_\infty\le 1$ for all $\gamma\ge 0$ and therefore, under Assumption $\mu<\mu_0$, Equation~\eqref{eq::NF-intro} with a response function satisfying Hypothesis~\ref{hyp::on the response function} and with the input $I$ defined by \eqref{eq::MacKay rays we control} with $\varepsilon>0$ will always reproduce the MacKay effect associated with ``MacKay rays''  thanks to Proposition~\ref{pro::2-nonlinear MacKay}.

In the case where $\gamma_0$ is finite, one has the following.
\begin{theorem}\label{thm::nonlinear MacKay}
	Let $L>0$.	If $\gamma_0$ defined by \eqref{eq:: gamma 0} is finite, there exists $\delta>0$ such that the stationary solution to Equation~\eqref{eq::NF-intro} with a response function satisfying Hypothesis~\ref{hyp::on the response function} and with the input $I$ defined by \eqref{eq::MacKay rays we control} with $\varepsilon>0$ small and $|\gamma-\gamma_0|\le\delta$ has the same zeroes structure as in the linear case in $[0,L]\times\R$, under Assumption $\mu<\mu_0$. In particular, it replicates the MacKay effect associated with the ``MacKay rays''.
\end{theorem}
\begin{proof}
	Let $\varepsilon>0$ be small and $\gamma_0$ defined by \eqref{eq:: gamma 0} be finite. On one hand, by definition of $\gamma_0$ and Proposition~\ref{pro::1-nonlinear MacKay}, the stationary solution $a_I(x_1,x_2)$ to Equation~\eqref{eq::NF-intro} with a response function satisfying Hypothesis~\ref{hyp::on the response function} and with the input $I$ defined by \eqref{eq::MacKay rays we control} with $\gamma=\gamma_0$ has a discrete and countable zero-level set with respect to each of its variables $x_1>0$ and $x_2\in\R$. On the other hand, one has for all $\gamma\ge 0$, $\Psi(\gamma P_F+\varepsilon v) = \Psi(\gamma P_F)+\varepsilon u_\gamma+o(\varepsilon)$ where  $u_\gamma\in L^\infty(\R^2)$ is the solution of $u_\gamma=v+\mu\omega\ast \left[f'(\Psi(\gamma P_F))u_\gamma\right]$. We known from Theorems~\ref{thm::stationary input in linear regime} and \ref{thm::characterization of properties on a_F} that $\Psi(\gamma P_F)$ has a discrete set of zeroes with respect to $x_2$ as $P_F$, and from Proposition~\ref{pro::2-nonlinear MacKay} that for all $\eta>0$, there exists $\delta>0$ such that, if $|\gamma-\gamma_0|\le\delta$ it holds $\|u_{\gamma}-u_{\gamma_0}\|_\infty\le\eta$. Therefore, since $u_{\gamma_0}$ has a discrete set of zeroes with respect to $x_1>0$, then the zeroes of the function $u_{\gamma}$ cannot accumulate at any of those zeroes in a finite interval, that is, the zeroes of both functions are distributed similarly in $[0, L]\times\R$ for all finite $L>0$.
\end{proof}
\begin{remark}
	Although a sigmoid nonlinearity such as $f(s)=\tanh(s)$ or $f(s) = \erf(s\sqrt{\pi}/2)$ does not satisfy the assumption $f(s) = s$ for $|s|\le 1$, it is almost linear in a small interval of the form $(-\varepsilon,\varepsilon)$, $\varepsilon>0$ in such a way that Theorem~\ref{thm::nonlinear MacKay} should be a theoretical explanation of why Equation~\eqref{eq::NF-intro} with this nonlinearity replicates the MacKay effect.
\end{remark}

\subsection{Numerical results for the visual MacKay effect}\label{ss::numerical results for the MacKay effect}
The numerical implementation is performed with Julia, where we coded retino-cortical map for visualising each experiment. Moreover, given a sensory input $I$, the associated stationary output $a_I$ is numerically implemented via an iterative fixed-point method. Following the convention adopted in \cite{ermentrout1979, bressloff2001} for geometric visual hallucinations, we present binary versions of these images, where black corresponds to positive values and white to negative ones as explained in Section~\ref{ss:: Binary pattern}. The reader may refer to \cite[Appendix~B]{tamekue:tel-04230895} for a toolbox that performs numerical results presented here.

The cortical data is defined on a square $(x_1,x_2)\in[-L, L]^2$, $L=10$ with steps $\Delta x_1 = \Delta x_2 = 0.01$. For the reproduction of the MacKay effect, parameters in the kernel $\omega$ given by \eqref{eq::connectivity} are $\kappa=1$, $2\pi^2\sigma_1^2=1$, and $2\pi^2\sigma_2^2=2$. We also choose $\mu:=1$. We collected some representative results in Figures~\ref{fig::mackay-rays-linear}, \ref{fig::mackay-target-linear}, \ref{fig::mackay-rays-nonlinear} and \ref{fig::mackay-target-nonlinear}. Here, we visualize the retinal representation obtained from the cortical patterns via the inverse retino-cortical map. In Figure~\ref{fig::mackay-rays-linear}, we exhibit the MacKay effect associated with the ``MacKay rays''. In this case, the sensory input is chosen as $I(x) = \cos(5\pi x_2)+\varepsilon H(2-x_1)$, where $\varepsilon = 0.025$ and $H$ being the Heaviside step function. Similarly, we exhibit in Figure~\ref{fig::mackay-target-linear} the MacKay effect associated with the ``MacKay target''. In this case, the sensory input is $I(x) = \cos(5\pi x_1)+\varepsilon (H(-x_2-9.75)+H(x_2-9.75)+H(0.25-|x_2|))$, where $\varepsilon = 0.025$ and $H$ being the Heaviside step function. We use a linear response function ($f(s)=s$) for the two figures. However, the phenomenon can be reproduced with any sigmoid function. See for instance, Figure~\ref{fig::mackay-rays-nonlinear}  and Figure~\ref{fig::mackay-target-nonlinear}.
\begin{figure}
	\centering
	\includegraphics[width =.9\linewidth]{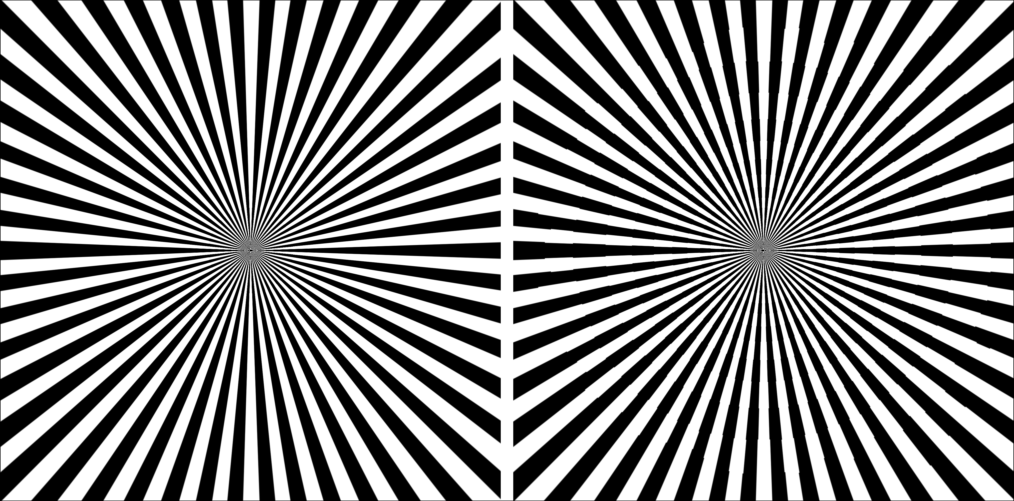}
	\caption{MacKay effect (\emph{right}) on the ``MacKay rays'' (\emph{left}). We use the linear response function $f(s)=s$. The sensory input is chosen as $I(x) = \cos(5\pi x_2)+\varepsilon H(2-x_1)$, $\varepsilon = 0.025$, where $H$ is the Heaviside step function.}
	\label{fig::mackay-rays-linear}
\end{figure}
\begin{remark}
	Although the Gaussian kernel is usually used in image processing and computer vision tasks due to its proximity to the visual system, it cannot replicate the MacKay effect if we use it as the kernel in Equation~\eqref{eq::NF-intro}. A physiological reason for this is that we used a one-layer model of NF equations. It is not then biologically realistic to model synaptic interactions with a Gaussian, which would model only excitatory-type interactions between neurons, see also Remark~\ref{rmk::about Gaussian kernel}. 
\end{remark}

\begin{figure}
	\centering
	\includegraphics[width = .9\linewidth]{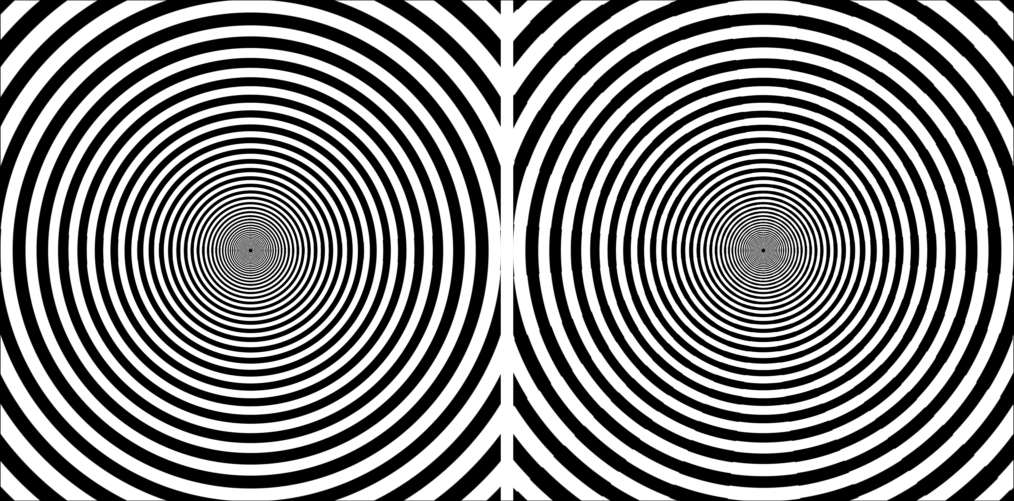}
	\caption{MacKay effect (\emph{right}) on the ``MacKay target'' (\emph{left}).We use the linear response function $f(s)=s$. The sensory input is $I(x) = \cos(5\pi x_1)+\varepsilon (H(-x_2-9.75)+H(x_2-9.75)+H(0.25-|x_2|))$, $\varepsilon = 0.025$, where $H$ is the Heaviside step function.}
	\label{fig::mackay-target-linear}
\end{figure}

\begin{figure}[ht!]
	\centering
	\includegraphics[width =.9\linewidth]{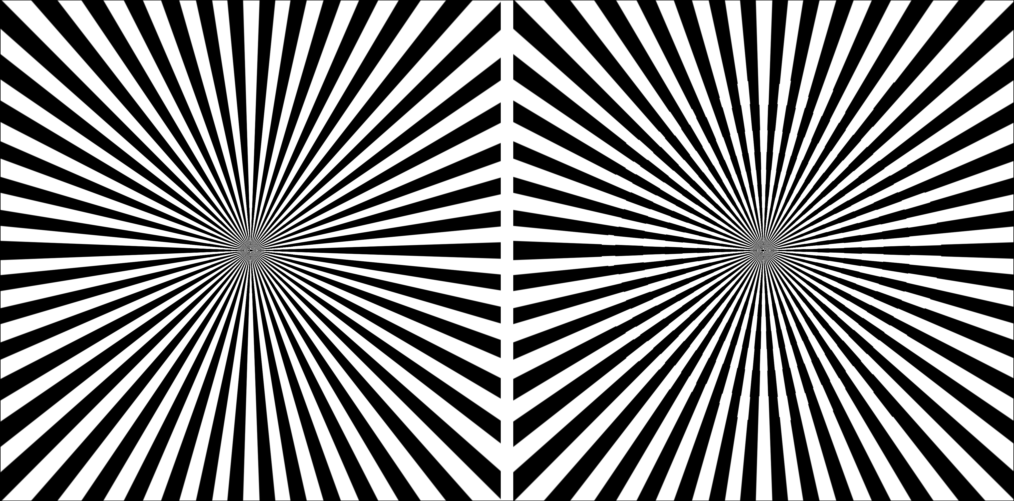}
	\caption{MacKay effect (\emph{right}) on the ``MacKay rays'' (\emph{left}). We use the nonlinear response function $f(s)=s/(1+|s|)$. The sensory input is chosen as $I(x) = \cos(5\pi x_2)+\varepsilon H(2-x_1)$, $\varepsilon = 0.025$, where $H$ is the Heaviside step function.}
	\label{fig::mackay-rays-nonlinear}
\end{figure}
\begin{figure}[ht!]
	\centering
	\includegraphics[width = .9\linewidth]{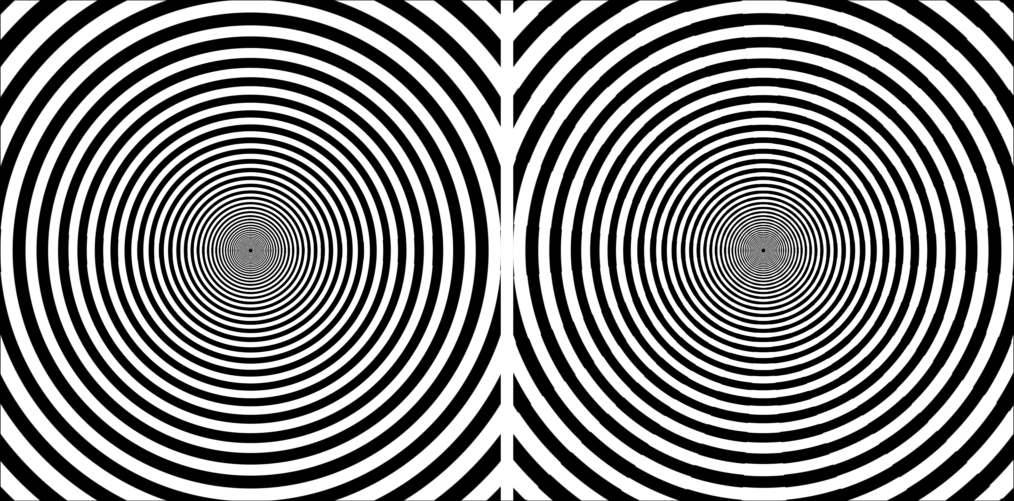}
	\caption{MacKay effect (\emph{right}) on the ``MacKay target'' (\emph{left}).We use the nonlinear response function $f(s)=s/(1+|s|)$. The sensory input is $I(x) = \cos(5\pi x_1)+\varepsilon (H(-x_2-9.75)+H(x_2-9.75)+H(0.25-|x_2|))$, $\varepsilon = 0.025$, where $H$ is the Heaviside step function.}
	\label{fig::mackay-target-nonlinear}
\end{figure}

						\section{Discussion} In this paper, we investigated the replication of visual illusions reported by MacKay \cite{mackay1957}, referred to as the visual MacKay effect. We have shown that these intriguing visual phenomena can be theoretically explained through a neural field model of Amari-type modelling the average membrane potential of V1 spiking neurons, which takes into account the sensory input from the retina. In our model equation, the sensory input stands to the V1 representation via the retino-cortical map of the visual stimulus employed in the MacKay experiment. 
						Assuming that the intra-neuron connectivity parameter is smaller than the threshold parameter where cortical patterns spontaneously emerge in V1 when sensory inputs from the retina do not drive its activity, we expounded a mathematical sound framework consisting of the input-output controllability of this equation. Then, performing a quantitative and qualitative study of the stationary output, we found that the MacKay effect is essentially a linear phenomenon, meaning the non-linear nature of the neural response does not play a role in its replication via our model equation. 
						
						Although our approach differs from that of Nicks \emph{et al.} \cite{nicks2021} in describing the MacKay-like effect (associated with regular sensory input), it agrees with the latter in emphasizing the role of inhibitory neurons in shaping the response of excitatory neurons to visual stimuli. This is consistent with the idea that inhibitory neurons play an important role in shaping the receptive fields of neurons in the visual cortex and that the interaction between excitatory and inhibitory neurons is crucial for visual processing. This new approach offers the advantage of accommodating any geometrical visual stimulus, particularly those localized in the visual field. We hope that this take on the question can serve as a foundation for future investigations, such as the theoretical replication of other psychophysical phenomena, including Billock and Tsou experiments \cite{billock2007}, the apparent motion in quartet stimulus \cite{giese1996}, the flickering wheel illusion \cite{Sokoliuk2013}, the spin in the enigma stimulus of Isia Léviant \cite{zeki1993,leviant1996}, or other psychophysical phenomena involving spontaneous cortical patterns such as the Barber pole, Café wall, Fraser spiral illusions \cite{fraser1908new,kitaoka2004contrast}, etc..
						
						
					\appendix
					\section{Equivariance of the input-output map with respect to the plane Euclidean group}\label{ss::action of E(2)}
					We discuss in this section the equivariance of the input to stationary output map $\Psi$ defined in \eqref{eq::map Psi} with respect to the plane Euclidean group.
					
					Let $\mathbf{E}(2)$ denote the Euclidean group, which is the symmetry group of $\R^2$. It is well known that (see, \cite[Chapter IV]{vilenkin1968} for instance) $\mathbf{E}(2)$ is the cross product of two-dimensional real line space $\R^2$ and $\mathcal{O}(2)$ the group of Euclidean rotations and reflections of this space, the so-called orthogonal group :
					$\mathbf{E}(2) = \R^2\rtimes\mathcal{O}(2)$.
					For any $g = (a,r)\in \mathbf{E}(2)$, one has $(a,r)\in\R^2\times\mathcal{O}(2)$ and the group property is the following
					$$\begin{cases}
						g_1\cdot g_2 = (a_1,r_1)\cdot(a_2,r_2) = (r_1a_2+a_1,r_1r_2),\cr
						g^{-1} = (-r^{-1}a,r^{-1}),\cr
						e = (0,\idty).
					\end{cases}$$
					Here, $g^{-1}$ is the inverse of $g = (a,r)\in \mathbf{E}(2)$, $e$ is the identity in $\mathbf{E}(2)$ and $\idty$ is the identity in $\mathcal{O}(2)$.
					
					\begin{definition}[Action of $\mathbf{E}(2)$ on $\R^2$]
						For $x\in\R^2$, the action of $g = (a,r)\in \mathbf{E}(2)$ on $\R^2$ is defined by
						$gx = rx+a$.
					\end{definition}
					\begin{definition}[Action of $\mathbf{E}(2)$ on $L^p(\R^2)$]
						We define the action of $\mathbf{E}(2)$ on $L^p(\R^2)$ by the representation $T:g\in \mathbf{E}(2)\longmapsto T_g\in \GL(L^p(\R^2))$
						such that, for all $v\in L^p(\R^2)$, it holds
						$$(T_gv)(x) = v(g^{-1}x), \qquad x\in\R^2.$$
						Here $\GL(L^p(\R^2))$ is the group of automorphism from $L^p(\R^2)$ to itself.
					\end{definition}
					
					We emphasise that the validity of the following proposition depends solely on the symmetry properties satisfied by the kernel $\omega$ rather than the nonlinear function $f$.  It remains valid whatever the shape (even linear, etc.) of the response function $f$.
					\begin{proposition}\label{pro::equivariance of the convolution operator}
						Let $\mu_0$ be defined by \eqref{eq::parameter mu_0}. If $\mu<\mu_0$, then, the map $\Psi$ defined in \eqref{eq::map Psi} and its inverse $\Psi^{-1}$ are $\mathbf{E}(2)$-equivariant, that is 
						\begin{equation}
							\Psi T_g = T_g \Psi
							\qquad\text{and}\qquad
							\Psi^{-1} T_g = T_g \Psi^{-1},
							\qquad \text{for any } g\in \mathbf{E}(2).
						\end{equation}
					\end{proposition}
					
					\begin{remark}\label{rmk::equivariance of the convolution operator}
						As a consequence of Proposition~\ref{pro::equivariance of the convolution operator} we have that the sensory input $I$ and the stationary output $\Psi(I)$ have the same symmetry subgroups $\Gamma\subset\mathbf{E}(2)$. For example,  $I$ depends solely on the $x_1$ variable if and only if the same is true for $\Psi(I)$.
					\end{remark}
					
					\begin{proof}[Proof of Proposition~\ref{pro::equivariance of the convolution operator}]
						We start by claiming that $\cQ(v):=\omega\ast f(v)$ is an $\mathbf{E}(2)$-equivariant operator from $L^p(\R^2)$ to itself.
						The fact that $\cQ$ is well-defined is a consequence of Lemma~\ref{Lem::nonlinear operator Q}. We thus need to show that $T_g\cQ = \cQ T_g$, for any $g = (a,r)\in \mathbf{E}(2)$. 
						Let $v\in L^p(\R^2)$ and $x\in\R^2$.
						On one hand, one has
						\begin{equation}\label{eq::inter1}
							(T_g(\cQ(v)))(x)=\cQ(v)(g^{-1}x)
							= \int_{\R^2}\omega(|g^{-1}x-y|)f(v(y))dy.
						\end{equation}
						On the other hand, one has
						$$
						(\cQ(T_gv))(x) =\int_{\R^2}\omega(|x-y|)(f(T_gv))(y)dy=\int_{\R^2}\omega(|x-y|)f(v(r^{-1}(y-a)))dy.
						$$
						Setting $z = r^{-1}(y-a)$, then $dy = |\det r|dz = dz$, since $r\in\mathcal{O}(2)$ and
						$$|x-rz-a| =  |r(r^{-1}(x-a)-z)| = |g^{-1}x-z|.$$
						It follows that
						\begin{equation}\label{eq::inter2}
							(\cQ(T_gv))(x) = \int_{\R^2}\omega(|g^{-1}x-z|)f(v(z))dz,
						\end{equation}
						which completes the proof of the claim by identifying \eqref{eq::inter1} and \eqref{eq::inter2}.
						
						To complete the proof of the statement, we need to show that $T_g\Psi = \Psi T_g$ and $T_g\Psi^{-1} = \Psi^{-1}T_g$ for any $g\in\mathbf{E}(2)$. This is equivalent to prove that for all $I\in L^p(\R^2)$, $T_g\Psi(I) = \Psi(T_gI)$ and $T_g\Psi^{-1}(I) = \Psi^{-1}(T_gI)$. It follows from the previous claim that
						$$
						T_g\Psi(I) = T_gI+\mu T_g\cQ(\Psi(I)) =  T_gI+\mu\cQ(T_g\Psi(I)).
						$$
						On the other hand, one has
						$$
						\Psi(T_g I) = T_gI+\mu\cQ(\Psi(T_g I)).
						$$
						So, by the uniqueness of the stationary state provided by Theorem~\ref{thm::existence of stationary input}, we obtain $T_g\Psi(I)=\Psi (T_g I)$. Arguing similarly, we prove that $\Psi^{-1}$ is also $\mathbf{E}(2)$-equivariant.
					\end{proof}
					
					\section{Complement results}
					\subsection{Complement results for the MacKay effect replication in the linear regime}\label{ss::complement MacKay linear regime}
					This section contains various complements used in Section~\ref{ss::theoretical MacKay effect} to describe the MacKay effect when the response function in Equation~\eqref{eq::NF-intro} is linear.  The first result is the following.
					
					\begin{theorem}\label{thm::main 2}
						Under the considerations of Remark~\ref{rmk::particular considerations},
						the kernel $K$ defined in \eqref{eq:: K} can be recast for all $x\in\R^{*}$ as
						\begin{eqnarray}\label{eq:: kernel K}
							\frac{K(x)}{2\sqrt{\pi}} &=& e^{-\pi |x|\sqrt{\frac{2\pi}{3}}}\cos\left(\frac{\pi}{12}+\pi|x|\sqrt{\frac{2\pi}{3}}\right)+\sum\limits_{k=1}^{\infty}\frac{e^{-\pi c_k|x|\sqrt{\frac{2\pi}{3}}}}{c_k}\cos\left(\frac{\pi}{12}+\pi c_k|x|\sqrt{\frac{2\pi}{3}}\right)\nonumber\\
							&&+\sum\limits_{k=1}^{\infty}\frac{e^{-\pi d_k|x|\sqrt{\frac{2\pi}{3}}}}{d_k}\sin\left(\frac{\pi}{12}-\pi d_k|x|\sqrt{\frac{2\pi}{3}}\right),
						\end{eqnarray}
						where
						\begin{equation}\label{eq:: c_k and d_k}
							c_k = \sqrt{1+6k}\qquad k\in\N\qquad\mbox{and}\qquad d_k = \sqrt{-1+6k},\qquad k\in\N^{*}.
						\end{equation}
					\end{theorem}
					\begin{figure}
						\centering
						\includegraphics[width=.4\linewidth]{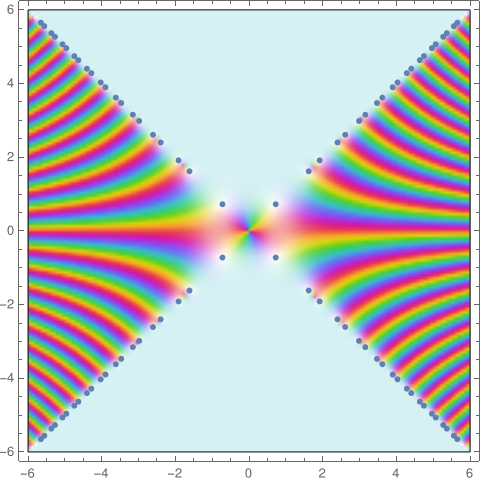}
						\caption{Zeroes in the complex plane of the exponential polynomial $h$ defined in \eqref{eq::exponential polynomial h}. Here $\kappa = \mu=1$, $2\pi^2\sigma_1^2=1$ and $2\pi^2\sigma_2^2=2$.}
						\label{poles_DoG}
					\end{figure}
					\begin{proof}
						We start by introducing for a fixed $x\in\R$, the function
						\begin{equation}\label{eq::def of g}
							g:z\in\C\mapsto g(z) =  e^{2i\pi z x}\frac{\widehat{\omega_1}(z)}{1-\widehat{\omega_1}(z)}=e^{2i\pi z x}\widehat{K}(z),\qquad\qquad \widehat{\omega_1}(z) = e^{-z^2}-e^{-2z^2}.
						\end{equation}
						We have that $g$ is a meromorphic function with simple poles (zeroes of the exponential polynomial $h$ defined in \eqref{eq::exponential polynomial h}) distributed as in Figure~\ref{poles_DoG} that we enumerate as $p_{k,\ell}$ and $q_{k,\ell}$ where $\ell\in\{0,\cdots,3\}$, by
						\begin{equation}\label{eq::poles of f, p}
							p_{k,\ell} = c_ke^{i\frac\pi 4}i^\ell\sqrt{\frac{\pi}{3}},\quad\quad k\in\N\quad\quad\mbox{and}\quad\quad
							q_{k,\ell} = d_ke^{i\frac\pi 4}i^\ell\sqrt{\frac{\pi}{3}},\qquad\qquad k\in\N^{*},
						\end{equation}
						where $c_k$ and $d_k$ are defined as in \eqref{eq:: c_k and d_k}. Since $\widehat{\omega}(p_{k,\ell}) = 1 = \widehat{\omega}(q_{k,\ell})$, we find the residues of $g$ to be given by
						\begin{equation}\label{eq::residue f, p}
							\Res(g,p_{k,\ell}) = -\frac{e^{i\frac\pi 4}i^\ell e^{i(-1)^\ell\frac\pi 3}}{2c_k\sqrt{\pi}}e^{2i\pi xp_{k,\ell}},\qquad k\in\N,\qquad
						\end{equation}
						\begin{equation}\label{eq::residue f, q}
							\Res(g,q_{k,\ell}) = \frac{e^{i\frac\pi 4}i^\ell e^{-i(-1)^\ell\frac\pi 3}}{2d_k\sqrt{\pi}}e^{2i\pi xq_{k,\ell}},\qquad k\in\N^{*}.
						\end{equation}
						We now fix $x>0$, and we let
						$$
						R_n:=\sqrt{n\pi},\qquad\qquad n\in\N^{*}.
						$$
						
						We consider the path $\Gamma_n$ straight along the real line axis from $-R_n$ to $R_n$ and then counterclockwise along a semicircle centred at $z=0$ in the upper half of the complex plane, $\Gamma_n = [-R_n,R_n]\cup C_n^{+}$, where $C_n^{+} = \{R_ne^{i\phi}\mid\phi\in[0,\pi]\}$. Then, by the residue Theorem, one has for all $n\in\N^{*}$,
						\begin{eqnarray}\label{eq::residue thm DoG}
							\int_{-R_n}^{R_n}g(\xi)d\xi+\int_{C_n^{+}}g(z)dz&=&2\pi i\sum\limits_{\ell=0}^{\ell=1}\sum\limits_{k=0}^{n-1}\Res(g,p_{k,\ell})+2\pi i\sum\limits_{\ell=0}^{\ell=1}\sum\limits_{k=1}^{n-1}\Res(g,q_{k,\ell})\nonumber\\
							&=&2\sqrt{\pi} e^{-\pi x\sqrt{\frac{2\pi}{3}}}\cos\left(\frac{\pi}{12}+\pi x\sqrt{\frac{2\pi}{3}}\right)+\nonumber\\
							&&2\sqrt{\pi}\sum\limits_{k=1}^{n-1}\frac{e^{-\pi c_kx\sqrt{\frac{2\pi}{3}}}}{c_k}\cos\left(\frac{\pi}{12}+\pi c_kx\sqrt{\frac{2\pi}{3}}\right)+\nonumber\\
							&&2\sqrt{\pi}\sum\limits_{k=1}^{n-1}\frac{e^{-\pi d_kx\sqrt{\frac{2\pi}{3}}}}{d_k}\sin\left(\frac{\pi}{12}-\pi d_kx\sqrt{\frac{2\pi}{3}}\right).
						\end{eqnarray}
						We set
						$$
						A_n(x) := \int_{C_n^{+}}g(z)dz.
						$$
						Then, one obtains,
						\begin{eqnarray}
							|A_n(x)|&\le& R_n\int_{0}^{\pi} e^{-2R_n\pi x\sin(\phi)}\vert\widehat{K}(R_ne^{i\phi})\vert d\phi\nonumber\\
							&=&\underbrace{R_n\int_{0}^{\frac\pi 4} e^{-2R_n\pi x\sin(\phi)}\vert\widehat{K}(R_ne^{i\phi})\vert d\phi}\limits_{J_1}+\underbrace{R_n\int_{\frac\pi 4}^{\frac{3\pi} 4} e^{-2R_n\pi x\sin(\phi)}\vert\widehat{K}(R_ne^{i\phi})\vert d\phi}\limits_{J_2}\nonumber\\
							& &+\underbrace{R_n\int_{\frac{3\pi} 4}^{\pi} e^{-2R_n\pi x\sin(\phi)}\vert\widehat{K}(R_ne^{i\phi})\vert d\phi}\limits_{J_3}.
						\end{eqnarray}
						Since $\vert\widehat{K}(R_ne^{i\phi})\vert\le 1$ for all $\phi\in[0, \pi]$, uniformly w.r.t. $n\in\N^{*}$, one has for all $x>0$,
						\begin{eqnarray}
							J_2:=R_n\int_{\frac\pi 4}^{\frac{3\pi} 4} e^{-2R_n\pi x\sin(\phi)}\vert\widehat{K}(R_ne^{i\phi})\vert d\phi&\le& R_n\int_{\frac\pi 4}^{\frac{3\pi} 4} e^{-2R_n\pi x\sin(\phi)}d\phi\nonumber\\
							&\le&\frac{\pi R_n}{2}e^{-R_n\pi x\sqrt{2}}\xrightarrow[n\to+\infty]{}0.
						\end{eqnarray}
						On the other hand, there exist a positive constant $C>0$ independent of $n\in\N^{*}$ ($C:=3/2$ is valid) such that for all $\phi\in[0, \pi]$, it holds
						$$
						\vert\widehat{K}(R_ne^{i\phi})\vert = \left|\frac{\widehat{\omega_1}(R_ne^{i\phi})}{1-\widehat{\omega_1}(R_ne^{i\phi})}\right|\le C|\widehat{\omega_1}(R_ne^{i\phi})|\le C \left(e^{-R_n^2\cos(2\phi)}+e^{-2R_n^2\cos(2\phi)}\right),\qquad \forall n\in\N^{*}.
						$$
						Since $\cos(2\phi)\ge -\frac{4}{\pi}\phi+1$ for all $\phi\in[0,\pi/4]$, one deduces
						\begin{eqnarray}
							J_1+J_3&\le& 2R_n \int_{0}^{\frac\pi 4} e^{-2R_n\pi x\sin(\phi)}\vert\widehat{K}(R_ne^{i\phi})\vert d\phi\le 2CR_n\int_{0}^{\frac\pi 4}e^{-R_n^2\cos(2\phi)}d\phi\nonumber\\
							&\le&4CR_n e^{-R_n^2}\int_{0}^{\frac\pi 4}e^{\displaystyle\frac{4}{\pi}R_n^2\phi}d\phi = \frac{C\pi}{R_n}\left[1-e^{-R_n^2 }\right]\xrightarrow[n\to+\infty]{}0.
						\end{eqnarray}
						To summarise, one has for all $x>0$,
						$$
						\int_{C_n^{+}}g(z)dz\xrightarrow[n\to+\infty]{}0.
						$$
						By taking the limit as $n\to+\infty$ in \eqref{eq::residue thm DoG} we find for all $x>0$,
						\begin{eqnarray}
							\frac{K(x)}{2\sqrt{\pi}}&=& e^{-\pi x\sqrt{\frac{2\pi}{3}}}\cos\left(\frac{\pi}{12}+\pi x\sqrt{\frac{2\pi}{3}}\right)+\sum\limits_{k=1}^{+\infty}\frac{e^{-\pi c_kx\sqrt{\frac{2\pi}{3}}}}{c_k}\cos\left(\frac{\pi}{12}+\pi c_kx\sqrt{\frac{2\pi}{3}}\right)\nonumber\\
							&&+\sum\limits_{k=1}^{+\infty}\frac{e^{-\pi d_kx\sqrt{\frac{2\pi}{3}}}}{d_k}\sin\left(\frac{\pi}{12}-\pi d_kx\sqrt{\frac{2\pi}{3}}\right).
						\end{eqnarray}
						Finally, the result follows at once since $K$ is an even function.
					\end{proof}
					\begin{remark}
						Since the kernel $K$ is even on $\R$, we will restrict its study to $\R_{+}$.
					\end{remark}
					
					In what follows, we aim to prove that $K$ admits a discrete and countable set of zeroes on $\R_{+}^{*}$. It is a consequence of the following.
					\begin{lemma}\label{Lem:estim0} For all $x\in \R_{+}^{*}$, it holds that
						\begin{equation}\label{eq:dev0}
							\frac{e^{\pi x \sqrt{\frac{2\pi}{3}}}K(x)}{2\sqrt{\pi}}=\cos\left(\frac{\pi}{12}+\pi x\sqrt{\frac{2\pi}{3}}\right)+\frac{S(x)}{x},
						\end{equation}
						where
						\begin{equation}\label{eq:reste0}
							\vert S(x)\vert\leq \frac{\sqrt{6}}{3\pi^2}.
						\end{equation}
						Moreover, the derivative of $K$ satisfies
						\begin{equation}\label{eq:dev1}
							\frac{\sqrt{3}e^{\pi x\sqrt{\frac{2\pi}{3}}}K'(x)}{4\pi^2} = -\sin\left(\frac{\pi}{3}+\pi x\sqrt{\frac{2\pi}{3}}\right)+T(x),
						\end{equation}
						where
						\begin{equation}\label{eq:reste1}
							\vert T(x)\vert\leq \frac {1+\pi x\sqrt{\frac{2\pi}{3}}}{\pi^3x^2}.
						\end{equation}
					\end{lemma}
					\begin{proof} Let $x>0$,
						one starts with the equation
						\[\frac{K(x)}{2\sqrt{\pi}}=e^{-\pi x \sqrt{\frac{2\pi}{3}}}\cos\left(\frac{\pi}{12}+\pi x\sqrt{\frac{2\pi}{3}}\right)+R_1(x)+R_2(x),
						\]
						where $R_1(x)=\sum\limits_{k=1}^{+\infty}r_1(k)\cos\left(\frac{\pi}{12}+\pi c_kx\sqrt{\frac{2\pi}{3}}\right)$ and $R_2(x)=\sum\limits_{k=1}^{+\infty}r_2(k)\sin\left(\frac{\pi}{12}-\pi d_kx\sqrt{\frac{2\pi}{3}}\right)$. The functions $r_1$ and $r_2$ are defined on $\mathbb R_+$ and $[1/3,+\infty)$ respectively by
						\[r_1(t)=\frac{ e^{-A\sqrt{1+6t}}}{\sqrt{1+6t}},\qquad\qquad r_2(t)=\frac{ e^{-A\sqrt{-1+6t}}}{\sqrt{-1+6t}}\qquad \text{ with }\quad A=\pi x \sqrt{\frac{2\pi}{3}}.
						\]
						Since $r_1$ is decreasing on $\mathbb R_+$ one deduces that
						\[\vert R_1(x)\vert\leq \sum\limits_{k=1}^{+\infty}r_1(k)\leq \sum\limits_{k=1}^{+\infty}\int_{k-1}^kr_1(t)\,dt=\int_0^\infty r_1(t)\,dt=\int_0^\infty  e^{-A\sqrt{1+6t}}\frac{dt}{\sqrt{1+6t}}=\frac{e^{-A}}{3A}.
						\]
						The same argument gives the same inequality for $\vert R_2(x)\vert$
						and inequality \eqref{eq:reste0} follows at once. On the other hand, it is  straightforward to observe that  the sum  $S(x)$ in  \eqref{eq:dev0} is uniformly (normally in fact) convergent on $(-\infty,-B]\cup[B,+\infty)$ for all $B>0$. Thus, after derivation under the sum, one finds for all $x>0$,
						\begin{eqnarray*}
							\frac{\sqrt{3}e^{\pi x\sqrt{\frac{2\pi}{3}}}K'(x)}{4\pi^2} &=&-\sin\left(\frac{\pi}{3}+\pi x\sqrt{\frac{2\pi}{3}}\right)-e^{\pi x\sqrt{\frac{2\pi}{3}}}\sum\limits_{k=1}^{\infty}e^{-\pi c_k x\sqrt{\frac{2\pi}{3}}}\sin\left(\frac{\pi}{3}+\pi c_kx\sqrt{\frac{2\pi}{3}}\right)\nonumber\\
							&&-e^{\pi x\sqrt{\frac{2\pi}{3}}}\sum\limits_{k=1}^{\infty}e^{-\pi d_k x\sqrt{\frac{2\pi}{3}}}\sin\left(\frac{\pi}{3}-\pi d_kx\sqrt{\frac{2\pi}{3}}\right)\nonumber\\
							&=&-\sin\left(\frac{\pi}{3}+\pi x\sqrt{\frac{2\pi}{3}}\right)+T(x),
						\end{eqnarray*}
						where
						$$
						|T(x)|\le e^{\pi x\sqrt{\frac{2\pi}{3}}}\sum\limits_{k=1}^{\infty}\left(e^{-\pi c_k x\sqrt{\frac{2\pi}{3}}}+e^{-\pi d_k x\sqrt{\frac{2\pi}{3}}}\right)\le 2e^{\pi x\sqrt{\frac{2\pi}{3}}}\sum\limits_{k=1}^{\infty}e^{-\pi d_kx\sqrt{\frac{2\pi}{3}}},
						$$
						since $c_k\ge d_k$ for all $k\ge 1$.
						But one has
						$$
						\sum\limits_{k=1}^{\infty}e^{-\pi d_k x\sqrt{\frac{2\pi}{3}}}\le\sum\limits_{k=1}^{\infty}\int_{k-\frac 2 3}^{k}e^{-\pi x\sqrt{-1+6t} \sqrt{\frac{2\pi}{3}}}dt=\int_{\frac 1 3}^{\infty}e^{-\pi x\sqrt{-1+6t} \sqrt{\frac{2\pi}{3}}}dt=\frac {1+\pi x\sqrt{\frac{2\pi}{3}}}{2\pi^3x^2}e^{-\pi x\sqrt{\frac{2\pi}{3}}},
						$$
						so that inequality  \eqref{eq:reste1} follows at once and completes the proof of the lemma.
					\end{proof}
					\begin{proposition}\label{prop::zeroes of K}
						Let $(x_k)_{k\in\N^{*}}$ and $(y_k)_{k\in\N^{*}}$ denote the sequences of zeroes and extrema of the function $x\mapsto \cos(\pi/12+\pi x\sqrt{2\pi/3})$ on $\R_{+}^{*}$ respectively.
						There exists   $(z_k)_{k\in\N^{*}}$, sequence of zeroes of $K$ in $\R_{+}^{*}$ such that $z_k$ is the unique zero of $K$ in the interval $I_k:=]y_k,y_{k+1}[$ for all $k\in\N^{*}$ and
						\begin{equation}\label{eq::asymptotic of zeroes of K}
							|x_{k+1}-z_k|\le\frac{\sqrt{3}}{\pi\sqrt{2\pi}}\arcsin\left(\frac{8}{\pi(12k-1)}\right),\qquad\qquad \forall k\in\N^{*}.
						\end{equation}
					\end{proposition}
					\begin{proof}
						We fix $k\in\N^{*}$, then one has
						$$
						|S(y_k)|\le\frac{2}{\pi\sqrt{6\pi}y_k} = \frac{8}{\pi(12k-1)}\le\frac{8}{11\pi}<1,
						$$
						by Lemma~\ref{Lem:estim0}. One deduces that
						$$
						e^{\pi y_k \sqrt{\frac{2\pi}{3}}}\frac{K(y_k)}{2\sqrt{\pi}} = (-1)^k+S(y_k)\begin{cases}
							<0,\quad\mbox{if }\quad k \quad\mbox{is odd},\cr
							>0,\quad\mbox{if }\quad k\quad\mbox{is even}.
						\end{cases}
						$$
						It follows that $K$ admits at least one zero $z_k$ in the interval $I_k$ by the intermediate value theorem. Let us prove that $z_k$ is the unique zero in this interval.  We let $\widetilde{z}$ be an arbitrary zero of $K$ in the interval $I_k$ and set $e_k:=\widetilde{z}-x_{k+1}$. Then one has by Lemma~\ref{Lem:estim0}
						\begin{equation}\label{eq::S de z tilde}
							S(\widetilde{z}) = -\cos\left(\frac{\pi}{12}+\pi \widetilde{z}\sqrt{\frac{2\pi}{3}}\right)=(-1)^k\sin\left(\pi e_k\sqrt{\frac{2\pi}{3}}\right),
						\end{equation}
						and
						\begin{equation}\label{eq::intermediate}
							\left| \sin\left(\pi e_k\sqrt{\frac{2\pi}{3}}\right) \right|\le\frac{2}{\pi\sqrt{6\pi}\widetilde{z}}\le\frac{2}{\pi\sqrt{6\pi}y_k}\le\frac{2}{\pi\sqrt{6\pi}y_1} = \frac{8}{11\pi}.
						\end{equation}
						On the other hand, using \eqref{eq::S de z tilde} and trigonometric identity for sine, one obtains
						\begin{eqnarray}
							\frac{\sqrt{3}e^{\pi \widetilde{z}\sqrt{\frac{2\pi}{3}}}K'(\widetilde{z})}{4\pi^2} &=&-\sin\left(\frac{\pi}{12}+\pi \widetilde{z}\sqrt{\frac{2\pi}{3}}+\frac\pi 4\right)+T(\widetilde{z})\nonumber\\
							&=&\frac{(-1)^{k+1}}{\sqrt{2}}\cos\left(\pi e_k\sqrt{\frac{2\pi}{3}}\right)+\frac{1}{\sqrt{2}}S(\widetilde{z})+T(\widetilde{z}).
						\end{eqnarray}
						By using \eqref{eq::S de z tilde}, \eqref{eq::intermediate} and  \eqref{eq:reste1} one finds
						$$
						\cos\left(\pi e_k\sqrt{\frac{2\pi}{3}}\right)\ge \sqrt{1-\frac{8}{11\pi}} >\sqrt{1-\frac 1 2}=\frac{1}{\sqrt{2}},
						$$
						and 
						$$
						\left|\frac{1}{\sqrt{2}}S(\widetilde{z})+T(\widetilde{z})\right|\le\frac{1}{\sqrt{2}}\frac{8}{11\pi}+ \frac {1+\pi y_1\sqrt{\frac{2\pi}{3}}}{\pi^3y_1^2}<\frac{1}{2}.
						$$
						It follows that
						$$
						K'(\widetilde{z})\begin{cases}
							>0,\quad\mbox{if }\quad k \quad\mbox{is odd},\cr
							<0,\quad\mbox{if }\quad k\quad\mbox{is even}.
						\end{cases}
						$$
						
						Let $\widetilde{z}$ and $\widetilde{z}'$ be successive zeroes of $K$ in $I_k$ and assume that $k$ is odd to be fixed. Then $K'(\widetilde{z})>0$ and $K'(\widetilde{z}')>0$. By Rolle's theorem, there exists $\widetilde{z}''\in(\widetilde{z},\widetilde{z}')$ such that $K(\widetilde{z}'')=0$ and $K'(\widetilde{z}'')<0$, which is a contradiction of the fact that any zero $\widetilde{z}$ in $I_k$ satisfies $K'(\widetilde{z})>0$. Thus, $z_k$ is the unique zero of $K$ in the interval $I_k$. Finally, inequality \eqref{eq::intermediate} applied with $\widetilde{z}=z_k$ leads to inequality \eqref{eq::asymptotic of zeroes of K}, and this completes the proof of the proposition.
					\end{proof}
					\begin{remark}\label{rmk::about Gaussian kernel}
						Suppose we model the interaction of V1 neurons in Equation~\eqref{eq::NF-intro} with a Gaussian kernel $\omega$. In that case, we will obtain that the associated kernel $\widehat{K}$ defined in \eqref{eq::def of g} has two isolated poles located on the imaginary axis of the complex plane. The zero-order terms which dominate the expansion of $K$ given by \eqref{eq:: kernel K} are only an exponential decreasing function without a cosine multiplicative factor. Therefore, the kernel $K$ will never have infinitely many discrete distributed zeroes.
					\end{remark}
					
					\subsection{Miscellaneous complements}\label{ss::complements for nonlinear Mackay effect}
					
					Some of the results provided in this section were used in Section~\ref{ss::theoretical MacKay effect nonlinear} to describe the MacKay effect when the response function in Equation~\eqref{eq::NF-intro} is nonlinear.
					
					We recall from Theorem~\ref{thm::existence of stationary input} that, given $1\le p\le\infty$ and $I\in L^p(\R^2)$, then for any $a_0\in L^p(\R^2)$, the initial value Cauchy problem associated with Equation~\eqref{eq::NF-intro} has a unique solution $a\in X_p$. It is implicitly given for all $x\in\R^2$, and every $t\ge 0$ by 
					\begin{equation}\label{eq::Variation of constant formula NF}
						a(x,t) = e^{- t}a_0(x)+\left(1-e^{- t}\right)I(x)+\mu\int_{0}^{t}e^{-(t-s)}(\omega\ast f(a))(x,s)ds.
					\end{equation}
					
					Given $I\in L^\infty(\R^2)$, the following theorem improves the upper bound of the $L^\infty$-norm of the stationary state $a_I\in L^\infty(\R^2)$ provided in \eqref{eq::upper bound on Psi_I}.
					\begin{theorem}\label{thm::main theorem 3}
						Let $a_0\in L^\infty(\R^2)$, $I\in L^\infty(\R^2)$ with $\|I\|_\infty = 1$ and $a\in X_\infty$ be the solution of \eqref{eq::NF-intro}. It holds
						\begin{equation}\label{eq::bound on limsup of a}							\limsup\limits_{t\rightarrow+\infty}\|a(\cdot,t)\|_\infty\le g_1,
						\end{equation}
						where $g_1>0$ is the smaller fixed point of the following function
						\begin{equation}\label{eq::function g}
							g:x\in\R\longmapsto 1+\frac{\mu}{\mu_0}f(x)\in\R_{+}^{*}.
						\end{equation}
					\end{theorem}
					\begin{proof}
						We start by using \eqref{eq::Variation of constant formula NF}, \eqref{eq::nonlinear map Q 3} and Minkowski's inequality to obtain for a.e. $x\in\R^2$ and every $t\ge 0$,
						\begin{equation}\label{eq::bound on absolute value of a}
							|a(x,t)|\le e^{- t}\|a_0\|_{L^\infty}+(1-e^{- t}) +\frac {\mu }{\mu_0}(1-e^{- t}).
						\end{equation}
						Letting $t\to\infty$ in the last inequality, we find $V_\infty:=\limsup\limits_{t\rightarrow+\infty}\|a(\cdot,t)\|_\infty\le 1+\mu/\mu_0,$
						showing in particular that $V_\infty<\infty$. It follows that
						\begin{equation}\label{eq:: traduction of limsup}
							\forall\varepsilon>0,\;\exists T_\varepsilon>0\;\mbox{s.t.},\;\forall t\ge T_\varepsilon,\;\; \|a(\cdot,t)\|_\infty\le V_\infty+\varepsilon.
						\end{equation}
						Applying the variation of constants formula  \eqref{eq::Variation of constant formula NF}, starting at $T_\varepsilon>0$, one deduces for every $t>T_\varepsilon$ that
						\begin{eqnarray}
							\|a(\cdot,t)\|_\infty&\le&e^{-(t-T_\varepsilon)}\|a(\cdot,T_\varepsilon)\|_\infty+\left(1-e^{-(t-T_\varepsilon)}\right)+\mu \|\omega\|_1\int_{T_\varepsilon}^{t}e^{-(t-s)}f(\|a(\cdot,s)\|_\infty)ds\nonumber\\
							&\le& e^{-(t-T_\varepsilon)}(V_\infty+\varepsilon)+1+\frac{\mu}{\mu_0}f(V_\infty+\varepsilon).
						\end{eqnarray}
						Letting respectively $t\rightarrow\infty$ and $\varepsilon\rightarrow 0$ in the preceding inequality we find
						\begin{equation}\label{eq::bound of limitsup}
							V_\infty\le 1+\frac{\mu}{\mu_0}f(V_\infty).
						\end{equation}
						Let $(u_n)_n$ be the real sequence defined by
						\begin{equation}\label{eq::real sequence}
							u_0 = V_\infty,\qquad u_{n+1} = g(u_n),\qquad\forall n\ge 1.
						\end{equation}
						Then $(u_n)_n$ is a bounded and non-decreasing sequence. The boundedness of $(u_n)_n$ follows from the boundedness\footnote{Notice that in the case where the response function $f$ is only Lipschitz continuous (with the Lipschitz constant equal to $f'(0)=1$) but not bounded, the sequence $(u_n)_n$ is still bounded, via
							$$|u_n|\le V_\infty+\frac{\mu_0}{\mu_0-\mu},\qquad\forall n\in\N.$$} of the sigmoid function $f$. Let us prove by induction that the sequence $(u_n)_n$ is increasing. Due to the inequality \eqref{eq::bound of limitsup}, one has
						$$
						u_1 = g(u_0) = 1+\frac{\mu}{\mu_0}f(u_0) = 1+\frac{\mu}{\mu_0}f(V_\infty)\ge V_\infty = u_0.$$
						If $u_n\ge u_{n-1}$ then, since $f$ is non-decreasing, one obtains
						$$
						u_{n+1} = g(u_n) = 1+\frac{\mu}{\mu_0}f(u_n)\ge 1+\frac{\mu}{\mu_0}f(u_{n-1}) = g(u_{n-1})=u_n,
						$$
						showing that $(u_n)_n$ is a non-decreasing sequence. The monotone convergence and fixed point Theorems, we have that $(u_n)_n$ converges to the smaller fixed point $g_1>0$ of the function $g$, and \eqref{eq::bound on limsup of a} follows.
					\end{proof}
					
					Let $1\le p\le\infty$, we introduce for every $I\in L^p(\R^2)$, the map $\Phi_I:L^p(\R^2)\mapsto L^p(\R^2)$ defined for all $v\in L^p(\R^2)$ by
					\begin{equation}\label{eq::map Phi}
						\Phi_I(v)=	I+\mu\omega\ast f(v).
					\end{equation}
					\begin{theorem}\label{thm::derivative of Psi}
						Let $1< p\le\infty$. If $\mu<\mu_0$, then $\Psi$ belongs to $C^1(L^p(\R^2);L^p(\R^2))$ and the differential at $I\in L^p(\R^2)$ is given by
						\begin{equation}\label{eq::differential of Psi}
							D\Psi(I)h = (\idty-D\Psi(I))^{-1}h,\qquad\forall h\in L^p(\R^2).
						\end{equation}
					\end{theorem}
					
					The proof of Theorem~\ref{thm::derivative of Psi} is a consequence of the following two lemmas.
					\begin{lemma}\label{lem:differential of map Phi}
						Let $1< p\le\infty$, and $I\in L^p(\R^2)$. Then for every  $\mu>0$, the map $\Phi_I$ belongs to the space $C^1(L^p(\R^2);L^p(\R^2))$ and the differential at $v\in L^p(\R^2)$ is given by
						\begin{equation}\label{eq::Gateau differential of Phi}
							(D\Phi_I(v)h)(x) = \mu\int_{\R^2}\omega(x-y)f'(v(y))h(y)dy,\qquad\forall h\in L^p(\R^2),\;x\in\R^2.
						\end{equation}
						Moreover, it  holds
						\begin{equation}\label{eq::norm of differential of Phi}
							\|D\Phi_I(v)\|_{\mathscr{L}(L^p(\R^2))}\le\frac{\mu}{\mu_0},\qquad\forall v\in L^p(\R^2).
						\end{equation}
					\end{lemma}
					\begin{proof}
						It is straightforward to show that for all $1\le p\le\infty$, and $I\in L^p(\R^2)$, the map $\Phi_I$ is Gateau-differentiable at every $v\in L^p(\R^2)$, the Gateau-differential is given for every $h\in L^p(\R^2)$ by \eqref{eq::Gateau differential of Phi}. Since $f'$ is bounded by $1$, we find
						\begin{equation}\label{eq::norm of Gateau-differential of Phi}
							\|D\Phi_I(v)h\|_p\le\frac{\mu}{\mu_0}\|h\|_p.
						\end{equation}
						Let us now show that for all $1< p\le\infty$, the Gateau-differential
						\begin{eqnarray}
							D\Phi_I:&L^p(\R^2)&\longrightarrow\mathscr{L}(L^p(\R^2))\nonumber\\
							&v&\longmapsto D\Phi_I(v),
						\end{eqnarray}
						is continuous. To this end, let $(v_n)\subset L^p(\R^2)$ be a real sequence converging in the $L^p$-norm to $v\in L^p(\R^2)$. We want to prove that $D\Phi_I(v_n)$ converges to $D\Phi_I(v)$ in $\mathscr{L}(L^p(\R^2))$. Let $h\in L^p(\R^2)$ and set
						\begin{equation}
							R_n:x\in\R^2\longmapsto R_n(x) = \int_{\R^2}\omega(x-y)[f'(v_n(y))-f'(v(y))]h(y)dy.
						\end{equation}
						It is immediate to obtain the result when $p=\infty$. Indeed, since $f'$ is $\|f''\|_\infty$-Lipschitz continuous, one immediately gets that 
						\[
						\|R_n\|_{\infty}\le\|f''\|_\infty\|\omega\|_1\|v_n-v\|_\infty\|h\|_\infty,
						\]
						so that
						\begin{eqnarray}\label{eq::validity estimate infty}
							\|D\Phi_I(v_n)-D\Phi_I(v)\|_{\mathscr{L}(L^\infty (\R^2))} &=&\sup\limits_{\stackrel{h\in L^\infty (\R^2)}{\|h\|_\infty=1}}\|D\Phi_I(v_n)h-D\Phi_I(v)h\|_\infty\nonumber\\
							&\le& \mu\|f''\|_\infty\|\omega\|_1\|v_n-v\|_\infty\xrightarrow[n\to\infty]{} 0.
						\end{eqnarray}
						Let us turn to an argument for the cases $1< p<\infty$.  Since $(v_n)$ tends in the $L^p$-norm to $v\in L^p(\R^2)$, for every $\varepsilon>0$, there exists a positive integer $N\in\N$ such that for any $n\ge N$ it holds $\|v_n-v\|_p\le\varepsilon$.  In the following, we fix $\varepsilon>0$ and $N$ defined previously. For every $n\in\N$ such that $n\ge N$, we consider  $E_n:=\{y\in\R^2\mid |v_n(y)-v(y)|>\sqrt{\varepsilon}\}$, one has for every $x\in\R^2$,
						\begin{equation}\label{eq::separate}
							R_n(x) = \underbrace{\int_{\R^2\bs E_n}\omega(x-y)[f'(v_n(y))-f'(v(y))]h(y)dy}\limits_{\Lambda_1(x)}+\underbrace{\int_{E_n}\omega(x-y)[f'(v_n(y))-f'(v(y))]h(y)dy}\limits_{\Lambda_2(x)}.
						\end{equation}
						By Chebyshev's inequality, it holds that
						\begin{equation}\label{eq:cheb}
							|E_n|\leq \frac{\|v_n-v\|_p^p}{\varepsilon^{\frac{p}2}}\leq \varepsilon^{\frac{p}2},
						\end{equation}
						where $|E_n|$ denotes the Lebesgue measure of the measurable set $E_n\subset \R^2$.
						
						On one hand, using Hölder inequality and the fact that $f'$ is $\|f''\|_\infty$-Lipschitz continuous, one 
						has
						\[
						|\Lambda_1(x)|\le \|f''\|_\infty\sqrt{\varepsilon}\|\omega\|_1^{\frac{1}{q}}\left\{\int_{\R^2\bs E_n}|\omega(x-y)||h(y)|^pdy\right\}^{\frac{1}{p}},\qquad \forall x\in\R^2.
						\]
						Taking the $p$-th power on both sides of the above inequality and integrating it with variable $x$ over $\R^2$, we find, thanks to Fubini's theorem,
						\begin{equation}\label{eq::separate 1}
							\|\Lambda_1\|_p:=  \left\{\int_{\R^2} |\Lambda_1(x)|^p dx\right\}^{\frac{1}{p}}\le \|f''\|_\infty\|\omega\|_1\sqrt{\varepsilon}\|h\|_p.
						\end{equation}
						On the other hand,  using Hölder inequality and the fact that $f'$ is bounded by $1$, one has
						\begin{eqnarray*}
							|\Lambda_2(x)|^p&\le&2^p|E_n|^{\frac pq}\int_{E_n}|\omega(x-y)|^p|h(y)|^pdy,\qquad \forall x\in\R^2.
						\end{eqnarray*}
						Integrating the above inequality with variable $x$ over $\R^2$, we find, thanks to Fubini's theorem,
						\begin{equation}\label{eq::separate 2}
							\|\Lambda_2\|_p:=\left\{\int_{\R^2} |\Lambda_2(x)|^p dx\right\}^{\frac{1}{p}}\le 2|E_n|^{\frac 1q}\|\omega\|_p\|h\|_p\le 2\|\omega\|_p\|h\|_p(\sqrt{\varepsilon})^{p-1},
						\end{equation}
						where the last inequality is obtained thanks to \eqref{eq:cheb} and $p/q = p-1$.
						
						Taking now the $p$-th power on both sides of inequality~\eqref{eq::separate}, integrating it with variable $x$ over $\R^2$, applying Minkowski's inequality and using \eqref{eq::separate 1} and \eqref{eq::separate 2}, one gets
						\[
						\|R_n\|_p\le(\|f''\|_\infty\|\omega\|_1+2\|\omega\|_p)\max(\sqrt{\varepsilon}, (\sqrt{\varepsilon})^{p-1})\|h\|_p.
						\]
						Therefore,
						\begin{eqnarray}\label{eq::validity estimate}
							\|D\Phi_I(v_n)-D\Phi_I(v)\|_{\mathscr{L}(L^p(\R^2))} &=&\sup\limits_{\stackrel{h\in L^p(\R^2)}{\|h\|_p=1}}\|D\Phi_I(v_n)h-D\Phi_I(v)h\|_p\nonumber\\
							&\le&\mu(\|f''\|_\infty\|\omega\|_1+2\|\omega\|_p)\max(\sqrt{\varepsilon}, (\sqrt{\varepsilon})^{p-1}). 
						\end{eqnarray}
						Letting $\varepsilon$ tend to zero, one deduces that $D\Phi_I$ is continuous.
						Finally, using \eqref{eq::norm of Gateau-differential of Phi} we find for all $v\in L^p(\R^2)$,
						$$
						\|D\Phi_I(v)\|_{\mathscr{L}(L^p(\R^2))} = \sup\limits_{\stackrel{h\in L^p(\R^2)}{\|h\|_p=1}}\|D\Phi_I(v)h\|_p\le\frac{\mu}{\mu_0},
						$$
						completing the proof of the lemma.
					\end{proof}
					\begin{lemma}\label{lem:map G}
						Let $1< p\le\infty$. Under assumption $\mu<\mu_0$, the map
						\begin{eqnarray}\label{eq::map G}
							\cG:&L^p(\R^2)\times L^p(\R^2)&\longrightarrow L^p(\R^2)\nonumber\\
							&(I,a)&\longmapsto\cG(I,\Psi(I)) =  a-\Phi_I(a),
						\end{eqnarray}
						belongs to $C^1(L^p(\R^2)\times L^p(\R^2);L^p(\R^2))$ and the partial derivative $D_a\cG(I,a)$ is invertible in $\mathscr{L}(L^p(\R^2))$.
					\end{lemma}
					\begin{proof}
						Since $\Phi_I$ is differentiable at $a\in L^p(\R^2)$ for all $I\in L^p(\R^2)$, one has for all $(J,b)\in L^p(\R^2)\times L^p(\R^2)$,
						\begin{eqnarray*}
							\cG(I+J,a+b) &=& a+b-\Phi_{I+J}(a)-D\Phi_{I+J}(a) b+o(\|b\|_p)\nonumber\\
							&=&\cG(I,a)+(\idty-D\Phi_{I}(a))b-J+o(\|b\|_p).
						\end{eqnarray*}
						The map $L_{(I,a)}:L^p(\R^2)\times L^p(\R^2)\longrightarrow L^p(\R^2)$,
						$L_{(I,a)}(J,b)=(\idty-D\Phi_{I}(a))b-J$,
						is linear and bounded,
						$$
						\|L_{(I,a)}(J,b)\|_{ L^p(\R^2)}\le\left(1+\frac{\mu}{\mu_0}\right)\|(J,b)\|_{L^p(\R^2)\times L^p(\R^2)}.
						$$
						It follows that $\cG$ is differentiable at $(I,a)\in L^p(\R^2)\times L^p(\R^2)$ and
						$$
						D\cG(I,a)(J,b) = (\idty-D\Phi_{I}(a))b-J,\qquad\forall(J,b)\in L^p(\R^2)\times L^p(\R^2).
						$$
						We now show that the map $(I,a)\in L^p(\R^2)^2\mapsto D\cG(I,a)\in\mathscr{L}(L^p(\R^2)^2,L^p(\R^2))$ is continuous. Let $(I_1,a_1), (I_2,a_2) \in L^p(\R^2)\times L^p(\R^2)$. One has for all $(J,b)\in L^p(\R^2)\times L^p(\R^2)$,
						\begin{eqnarray*}
							\|D\cG(I_1,a_1)(J,b)-D\cG(I_2,a_2)(J,b)\|_p&=&\|D\Phi_{I_1}(a_1)b-D\Phi_{I_2}(a_2) b\|_p\nonumber\\
							&\le&\|D\Phi_{I_1}(a_1)-D\Phi_{I_1}(a_2)\|_{\mathscr{L}(L^p(\R^2))}\|(b,J)\|_{L^p(\R^2)^2}.
						\end{eqnarray*}
						It follows by Lemma \ref{lem:differential of map Phi} that
						\begin{eqnarray*}
							\|D\cG(I_1,a_1)-D\cG(I_2,a_2)\|_{\mathscr{L}(L^p(\R^2)^2,L^p(\R^2))}&\le&\|D\Phi_{I_1}(a_1)-D\Phi_{I_2}(a_2)\|_{\mathscr{L}(L^p(\R^2))}\\
							&\le&\mu\|f\|_\infty\|\omega\|_{1}\|(I_1,a_1)-(I_2,a_2)\|_{L^p(\R^2)^2},
						\end{eqnarray*}
						showing that $\cG$ belongs to $C^1(L^p(\R^2)\times L^p(\R^2);L^p(\R^2))$.
						Finally, if $I\in L^p(\R^2)$, $a_{I}:=\Psi(I)\in L^p(\R^2)$ then
						$
						D_a\cG(I,a_{I}) = \idty-D\Phi_{I}(a_{I}),
						$
						is invertible in $\mathscr{L}(L^p(\R^2))$ if $\mu<\mu_0$.
					\end{proof}
					
					We now can present the proof of Theorem~\ref{thm::derivative of Psi}.
					
					\begin{proof}[Proof of Theorem~\ref{thm::derivative of Psi}]
						Let $\mu<\mu_0$. For fixed $I\in L^p(\R^2)$, $a_{I}:=\Psi(I)\in L^p(\R^2)$, we have $\cG(I,a_{I}) = 0$, and $D_a\cG(I,a_{I})$
						is invertible in $\mathscr{L}(L^p(\R^2))$ by Lemma \ref{lem:map G}. It follows by the implicit function Theorem that there is an open neighbourhood $\cV$ of $I$ in $L^p(\R^2)$, an open neighbourhood $\cW$ of $a_{I}$ in $L^p(\R^2)$ and a map $\Sigma:\cV\rightarrow\cW$ of class $C^1$ such that the following holds
						$$
						(I\in\cV,\;a\in\cW\;\mbox{and }\;\cG(I,a) = 0)\Longleftrightarrow (I\in\cV \;\mbox{and }\;a = \Sigma(I)).
						$$
						Thereby, $\Psi(\cdot)|_{\cV} = \Sigma(\cdot)$ and then $\Psi$ is $C^1$ at $I$. Since $I\in L^p(\R^2)$ is arbitrary, it follows that $\Psi$ belongs to $C^1(L^p(\R^2);L^p(\R^2))$. Moreover, taking the derivative of $\cG(I, \Psi(I)) = 0$ at $I$, one deduces that
						\begin{equation}\label{eq::intermediate Phi}
							\left(\idty-D\Phi_{I}(\Psi(I))\right)(D\Psi(I)h) = h,\qquad\forall h \in L^p(\R^2).
						\end{equation}
						Thus, \eqref{eq::differential of Psi} is an immediate consequence of \eqref{eq::norm of differential of Phi}, \eqref{eq::intermediate Phi} and the Neumann expansion lemma.
					\end{proof}
					
					{
						\begin{lemma}
							\label{lem::differentiel-U}
							Let $1<p\le\infty$, $T>0$ and consider the solution $U$ to \eqref{eq::nonlinear semigroup}. Then, $D_{a_0}U(t,v)$ is a well-defined invertible operator for any $v\in L^p(\R^2)$ and every $0\le t\le T$. Moreover, it holds
							\begin{equation}\label{eq:differentiel-U}
								\begin{split}
									\|D_{a_0}U(t, v)-\idty\|_{\mathscr{L}(L^p(\R^2))}&\le t\left(1+\frac{\mu}{\mu_0}\right)e^{\left(1+\frac{\mu}{\mu_0}\right)t},\\
									\|\left[D_{a_0}U(t, v)\right]^{-1}-\idty\|_{\mathscr{L}(L^p(\R^2))}&\le t\left(1+\frac{\mu}{\mu_0}\right)e^{\left(1+\frac{\mu}{\mu_0}\right)t}.
								\end{split}
							\end{equation}
						\end{lemma}
						
						\begin{proof}
							From Lemma~\ref{lem:differential of map Phi}, one gets $DN(U(t,v))=-\idty+D\Phi_{0}(U(t,v))$. One deduces that for every $\operatorname{V}\in \mathscr{L}(L^p(\R^2))$, it holds
							\begin{eqnarray}\label{eq:uniform operators}
								\|DN(U(t,v))\operatorname{V}\|_{\mathscr{L}(L^p(\R^2))}&\le&\|DN(U(t,v))\|_{\mathscr{L}(L^p(\R^2))}\|\operatorname{V}\|_{\mathscr{L}(L^p(\R^2))}\nonumber\\
								&\le&\left(1+\frac{\mu}{\mu_0}\right)\|\operatorname{V}\|_{\mathscr{L}(L^p(\R^2))}.
							\end{eqnarray}
							It follows that $DN(U(t,v))$ is a bounded linear operator on $\mathscr{L}(L^p(\R^2))$. One also has that $t\in[0, T]\mapsto U(t,v)\in L^p(\R^2)$ is continuous and that $u\in L^p(\R^2)\mapsto D\Phi_{0}(u)\in\mathscr{L}(L^p(\R^2))$ is continuous thanks to Lemma~\ref{lem:differential of map Phi}. Together with \eqref{eq:uniform operators}, one deduces that $t\in[0, T]\mapsto DN(U(t,v))\in\mathscr{L}(\mathscr{L}(L^p(\R^2)))$ is continuous.
							
							By integrating \eqref{eq::nonlinear semigroup}, one gets for any $a_0\in L^p(\R^2)$ and $t\geq 0$,
							\begin{equation}\label{eq:diff-I-U}
								U(t,a_0)=a_0+\int_0^tN(U(s,a_0))\, ds. 
							\end{equation}
							It is immediate to see that $U(t,\cdot)$ is differentiable with respect to $a_0$ for every $t\geq 0$ and we use $D_{a_0}U(t,v)$ to denote the evaluation of this differential at any $v\in L^p(\R^2)$. Actually, differentiating \eqref{eq:diff-I-U} yields that $D_{a_0}U(\cdot,v)$ belongs to $C^{1}([0,T],\mathscr{L}(L^p(\R^2)))$ for every $T>0$ and satisfies 
							\begin{equation}\label{eq:uo 1}
								D_{a_0}U(t,v)=\idty+\int_{0}^{t}DN(U(s,v))D_{a_0}U(s,v)ds,\qquad 0\le t\le T.
							\end{equation}
							Differentiating the above equation, one sees that $t\mapsto DU_{a_0}(t,v)$
							is the solution of the homogeneous initial value problem 
							\begin{equation}\label{eq:abstract ivp}
								\begin{cases}
									\partial_tD_{a_0}U(t,v)&=DN(U(t,v))D_{a_0}U(t,v),\cr
									D_{a_0}U(0,v) &=\idty,
								\end{cases}
							\end{equation}
							defined for any $v\in L^p(\R^2)$ and every $0\le t\le T$. This system is nothing else but the differential of \eqref{eq::nonlinear semigroup} with respect to $a_0$. 
							By \cite[Theorem~5.2.-item (i)]{pazy2012}, one has
							\begin{equation}\label{eq:uo 2}
								\|D_{a_0}U(t,v)\|_{\mathscr{L}(L^p(\R^2))}\le e^{\left(1+\frac{\mu}{\mu_0}\right)t},\qquad 0\le t\le T.
							\end{equation}
							By combining \eqref{eq:uo 1} and \eqref{eq:uo 2}, one obtains the first inequality in \eqref{eq:differentiel-U}. On the other hand, \eqref{eq:uo 1} ensures that for $t\in[0, T]$ small enough,  $D_{a_0}U(s,v)$ is invertible in $\mathscr{L}(L^p(\R^2))$ for every $s\in[0, t]$  and it holds 
							\begin{equation}\label{eq:abstract ivp inverse}
								\begin{cases}
									\partial_s\left[D_{a_0}U(s, v)\right]^{-1}&=-\left[D_{a_0}U(s, v)\right]^{-1}DN(U(s,v)),\cr
									\left[D_{a_0}U(0, v)\right]^{-1} &=\idty.
								\end{cases}
							\end{equation}
							Arguing as above, one obtains that the homogeneous initial value problem \eqref{eq:abstract ivp inverse} admits a unique solution $\left[D_{a_0}U(\cdot,v)\right]^{-1}\in C^{1}([0,t],\mathscr{L}(L^p(\R^2)))$ given by
							\begin{equation}\label{eq:uo 1 inverse}
								\left[D_{a_0}U(s, v)\right]^{-1}=\idty-\int_{0}^{s}\left[D_{a_0}U(\tau, v)\right]^{-1}DN(U(\tau,v))d\tau,\qquad 0\le s\le t.
							\end{equation}
							Moreover, one has 
							\begin{equation}\label{eq:uo 2 inverse}
								\|\left[D_{a_0}U(s, v)\right]^{-1}\|_{\mathscr{L}(L^p(\R^2))}\le e^{\left(1+\frac{\mu}{\mu_0}\right)s},\qquad 0\le s\le t,
							\end{equation}
							and the second inequality in \eqref{eq:differentiel-U} follows for $t\in[0, T]$ small enough. Since \eqref{eq:abstract ivp} holds for every $0\le t\le T$, one can iterate this procedure in $[t, 2t],\cdots,$ to prove that $D_{a_0}U(t,v)$ is invertible in $\mathscr{L}(L^p(\R^2))$ for every $t\in[0, T]$ and that the second inequality in \eqref{eq:differentiel-U} holds.
						\end{proof}
						
					}

						\printbibliography
						
					\end{document}